\documentclass[11pt]{amsart}
\usepackage{
 amsmath, 
 amsxtra, 
 amsthm, 
 amssymb, 
 etex, 
 mathrsfs, 
 mathtools, 
 tikz-cd, 
 bbm,
 xr,
 comment,
 enumitem}
\usepackage[toc]{appendix} 
\usepackage{color}
\usepackage[color,matrix,all]{xy}
\usepackage{hyperref}
\usepackage{url}
\usepackage{tipa}
\usepackage{dsfont}
\usepackage{ textcomp }
\usepackage{stmaryrd} 
\usepackage{amssymb}
\usepackage{mathabx} 
\usepackage{amsmath} 
\usepackage{amscd}
\usepackage{amsbsy}
\usepackage{comment, enumerate}
\usepackage{color}
\usepackage{mathtools,caption}
\usepackage{tikz-cd}
\usepackage{longtable}
\usepackage[utf8]{inputenc}
\usepackage[OT2,T1]{fontenc}
\usepackage{changepage}
\usepackage{commath,enumerate}
\usetikzlibrary{shapes.geometric}

\DeclareMathOperator{\End}{End}

\DeclareMathOperator{\Gal}{Gal}

\DeclareMathOperator{\ord}{ord}

\newtheorem{theorem}{Theorem}[section]
\newtheorem*{theorem*}{Theorem}
\newtheorem{lemma}[theorem]{Lemma}

\newtheorem{proposition}[theorem]{Proposition}
\newtheorem{corollary}[theorem]{Corollary}
\newtheorem{defn}[theorem]{Definition}

\numberwithin{equation}{section}
\newtheorem{lthm}{Theorem} 

\newcommand{\fp}{\mathfrak p}

\mathcode`l="8000
\begingroup
\makeatletter
\lccode`\~=`\l
\DeclareMathSymbol{\lsb@l}{\mathalpha}{letters}{`l}
\lowercase{\gdef~{\ifnum\the\mathgroup=\m@ne \ell \else \lsb@l \fi}}%
\endgroup

\theoremstyle{remark}
\newtheorem{remark}[theorem]{Remark}
\newtheorem{example}[theorem]{Example}

\newcommand{\bE}{\mathbf{E}}
\newcommand{\EE}{\mathbb{E}}
\newcommand\EatDot[1]{}
\newcommand{\fF}{\mathfrak{F}}

\newcommand{\cG}{{\mathcal{G}}}

\newcommand{\cT}{{\mathcal{T}}}

\newcommand{\CC}{\mathbb{C}}

\newcommand{\fG}{\mathfrak{G}}
\newcommand{\Fp}{\mathbb{F}_p}

\newcommand{\cS}{{\mathcal{S}}}

\newcommand{\QQ}{\mathbb{Q}}
\newcommand{\ZZ}{\mathbb{Z}}
\newcommand{\cL}{\mathcal{L}}
\newcommand{\Qp}{\mathbb{Q}_p}
\newcommand{\Zp}{\mathbb{Z}_p}

\definecolor{Green}{rgb}{0.0, 0.5, 0.0}

\newcommand{\cO}{\mathcal{O}}

\newcommand{\cC}{\mathcal{C}}

\newcommand{\Q}{\mathbb{Q}}

\newcommand{\Z}{\mathbb{Z}}
\newcommand{\Aut}{\mathrm{Aut}}

\newcommand{\fL}{\mathfrak{L}}
\newcommand{\e}{\hfill$\Diamond$}

\renewcommand{\r}{\mathfrak r}
\renewcommand{\s}{\mathfrak s}
\renewcommand{\t}{\mathfrak t}
\newcommand{\cc}{\mathfrak c}

  \DeclareFontFamily{U}{wncy}{}
  \DeclareFontShape{U}{wncy}{m}{n}{<->wncyr10}{}
  \DeclareSymbolFont{mcy}{U}{wncy}{m}{n}
  \DeclareMathSymbol{\sha}{\mathord}{mcy}{"58}
  \DeclareMathSymbol{\zhe}{\mathord}{mcy}{"11}
\title{On ordinary isogeny graphs with level structures}

\newcommand{\cE}{\mathcal{E}}

\makeatletter
\let\@wraptoccontribs\wraptoccontribs
\makeatother

\usepackage{tikz,xcolor,hyperref}

\author[A. Lei]{Antonio Lei}
\address[Lei]{Department of Mathematics and Statistics\\University of Ottawa\\
150 Louis-Pasteur Pvt\\
Ottawa, ON\\
Canada K1N 6N5}
\email{antonio.lei@uottawa.ca}

\author[K. Müller]{Katharina Müller}
\address[Müller]{Institut für Theoretische Informatik, Mathematik und Operations Research\\Universität der Bundeswehr München\\ Werner-Heisenberg-Weg 39\\85577 Neubiberg\\Germany}
\email{katharina.mueller@unibw.de}


\begin{document}
\begin{abstract}
Let $l$ and $p$ be two distinct prime numbers. We study $l$-isogeny graphs of ordinary elliptic curves defined over a finite field of characteristic $p$, together with a level structure. Firstly, we show that as the level varies over all $p$-powers, the graphs form an Iwasawa-theoretic abelian $p$-tower, which can be regarded as a graph-theoretical analogue of the Igusa tower of modular curves.  Secondly, we study the structure of the crater of these graphs, generalizing previous results on volcano graphs. Finally, we solve an inverse problem of graphs arising from the crater of $l$-isogeny graphs with level structures, partially generalizing a recent result of Bambury, Campagna and Pazuki.
\end{abstract}

\maketitle

\section{Introduction}
\label{S: Intro}
Let $p$ be a fixed prime number. Let $X$ be a finite connected graph (in this article, we allow multiple edges and loops in a graph). In a series of articles, \cite{vallieres,vallieres2,vallieres3} Vallières and McGown--Vallières studied Iwasawa theory of the so-called abelian $p$-towers of graphs above $X$ (see Definition~\ref{def:tower}; note that we have replaced the prime $\ell$ in the aforementioned works by $p$ in this article). Let $(X_n)_{n\ge0}$ be such a tower and write $\kappa_n$ for the number of spanning trees of $X_n$. Then there exist integers $\mu$, $\lambda$ and $\nu$ such that
\[
\ord_p(\kappa_n)=\mu p^n+\lambda n+\nu
\]
for $n\gg0$. This can be regraded as the graph-theoretic analogue of the seminal formula of Iwasawa on class groups of sub-extensions inside a $\Zp$-extension of a number field proved in \cite{iwasawa73}.

As discussed in \cite{iwasawa69,mazurwiles83}, Iwasawa theory of number fields draws strong analogies with its function field counterpart in which one studies towers of Galois coverings of curves. One important example of such towers is the Igusa tower of modular curves initially studied in \cite{igusa}; see also \cite{katz-mazur,mazurwiles83,hida09}. Roughly speaking, an Igusa tower consists of 
\[
X_0\leftarrow X_1\leftarrow\cdots\leftarrow X_n\leftarrow X_{n+1}\leftarrow\cdots
\]
where $X_n$ is the modular curve classifying isomorphism classes of $(E,P)$, where $E$ is an elliptic curve over a  finite field  of characteristic $p$ and $P$ is a point on $E$ of order $p^n$.

The first goal of this article is to construct an explicit $\Zp$-tower of graph coverings arising from isogeny graphs of ordinary elliptic curves defined over a finite field $k$, whose characteristic is $p$. This gives a graph theoretical analogue of Igusa towers. Let $l$ be a prime number distinct from $p$ and let $m\ge0$, $N\ge1$ be integers. We define in \S\ref{S:define}  the $l$-isogeny graph  $G_N^m$ whose vertices consist of {$\overline{k}$-}isomorphism classes of pairs $(E,P)$, where $E$ is an ordinary elliptic curve defined over $k$ and $P$ is a point of $E(\overline{k})$ of order $Np^m$, and edges between two vertices are defined by $l$-isogenies. When $N=1$ and $m=0$, this recovers the  volcano graphs studied in \cite{kohel,sutherland,pazuki}. The graphs $G_N^m$ can be regarded as an enhancement of the volcano graphs via the addition of a $\Gamma_1(Np^m)$-level structure. Similar (but slightly different) graphs have been studied in \cite{gorenkassaei}. See in particular {Sections 2 and 3 in op. cit.} 

A priori, $G_N^m$ is a directed graph. We shall write $\tilde G_N^m$ for the undirected graph obtained from $G_N^m$ by ignoring the directions of the edges. The first main result of the present article is the following:

\begin{lthm}[{Corollary~\ref{cor:main}}]\label{thmA}
    Let $E$ be an elliptic curve representing a non-isolated vertex of $\tilde G^0_1$ (i.e., a vertex whose degree is strictly positive). Let $\tilde\cG_N^m$ denote the connected component of $\tilde G_N^m$ containing a vertex arising from $E$. Then there exists an integer $m_0$ such that the graphs $\left(\tilde \cG_N^{m_0+r}\right)_{r\ge0}$ form an abelian $p$-tower in the sense of Vallières and McGown--Vallières.
\end{lthm}
In Appendix \ref{app}, we explain how to realize such a tower of graph coverings as voltage assignment graphs when $N=1$. This may be of independent interest since  voltage assignments are used to define Iwasawa invariants of  abelian $p$-towers in \cite{vallieres,vallieres2,vallieres3}. 

The integer $m_0$ featured in Theorem~\ref{thmA} depends on the variation of the number of connected components in $G_N^m$ as $m$ increases. We show in Proposition~\ref{prop:components} that when $m$ is sufficiently large, the number of components in $
\tilde G_N^m$ stabilizes. Such stabilization is necessary to ensure that the coverings $
\tilde \cG_N^{m+r}/
\tilde \cG_N^m$ are Galois.

In the case of volcano graphs, each vertex is assigned a "level", depending on the endomorphism ring of the elliptic curve attached to the vertex. The level zero vertices form the "crater" of a volcano graph. In \cite[Proposition~3.14]{pazuki}, Bambury--Campagna--Pazuki gave a complete description of all possible craters. One may extend the concept of "levels" and "craters" to $G_N^m$ in a natural manner (see Definition~\ref{def:level}). In \S\ref{S:crater}, we give a detailed description of the crater of $G_N^m$; see in particular Remark~\ref{rk:conclude} and Proposition~\ref{prop:principal}. Several explicit examples are given throughout the section. The introduction of $\Gamma_1(Np^m)$-level structure has the advantage that we may avoid working with loops if we assume either $N$ or $m$ is sufficiently large (see Lemma~\ref{lem:degree-split case}). {While many of our results are direct analogues of those given in \cite{pazuki}, the absence of loops allows us to simplify some of the proofs.}

In \S\ref{S:inverse}, we define the so-called "abstract tectonic craters", which are graphs that have the same geometry as a connected component of the crater that we describe in \S\ref{S:crater} (see Definition~\ref{def:crater}). We prove a result on the inverse problem for such graphs. 

\begin{lthm}[{Theorem~\ref{thm:inverse}}]\label{thmB}
    Let  $G$ be an abstract tectonic crater. There exist infinitely many pairs of distinct primes $p $ and $l$, and nonnegative integers $N$ such that one of the connected components of the crater of the $l$-isogeny graph $G_N^1$ is isomorphic to $G$.
\end{lthm}

This can be regarded as a partial generalization of results in \cite{pazuki}, where the inverse problem for volcano graphs over $\Fp$ without level structure has been studied. The inverse problem without level structure is false when $k\ne \Fp$ because of connectedness issue (see \S5.1 of op. cit. for a detailed discussion). In the present paper, we consider connected components separately allowing us to avoid this issue. Furthermore, we  have the liberty to increase the level to simplify the structure of the graphs being studied. In particular, we do not recover results of \cite{pazuki} since the level is fixed to be $1$ in the aforementioned work.

\subsection*{Outlook}
In the setting of number fields, {questions on distributions of Iwasawa invariants attached to cyclotomic $\Zp$-extensions of imaginary quadratic fields have initially been studied in \cite{EJV}. More recently, similar questions on abelian number fields have been studied in \cite{delbourgoheiko}}. In \cite{DLRV}, questions on distributions of Iwasawa invariants attached to abelian $p$-towers of graphs were studied. Unlike the setting of number fields, the notion of cyclotomic extensions does not exist in the context of graphs. The towers given by Theorem~\ref{thmA} could potentially be a candidate of substitution for cyclotomic extensions. We plan to study how the Iwasawa invariants vary as $l$ and/or $p$ vary. Techniques developed in \cite{Ari} could potentially be adopted in this setting.

It may also be interesting to seek arithmetic interpretation of the $p$-adic zeta functions attached the towers given by Theorem~\ref{thmA}. One might naively hope that they could be related to  $p$-adic zeta functions of modular curves over $k$ originating from the Iwasawa theory of function fields. Results in  \cite{sugiyama} tell us that the latter are closely related to supersingular isogeny graphs. However,  we would not be able to construct abelian $p$-towers for supersingular isogeny graphs in the manner presented in this article because of the lack of $p$-power torsions on supersingular elliptic curves over finite fields of characteristic $p$. This suggests that fundamentally new ideas are required to establish links between abelian $p$-towers of ordinary isogeny graphs  with objects from function field Iwasawa theory.

In a different direction, we plan to study the the following inverse problem further generalizing Theorem~\ref{thmB}: Given a volcano graph $G$ with a tectonic crater (see Definition~\ref{def:crater}), can we find primes $p$ and $l$, a finite field $k$ of characteristic $p$ and an integer $N$ such that $G$ is a connected component of $G_N^m$ for some non-negative integer $m$? The additional difficulty in solving this question compared to Theorem~\ref{thmB} is that a volcano of depth $d>0$ (i.e a volcano not only consisting of a crater) does not leave any room for choosing the prime $l$, whereas our proof of Theorem~\ref{thmB} depends crucially on using Tchebotarev's Theorem to choose  $l$. One could hope to resolve this problem by choosing the imaginary quadratic field $K$ appropriately -- similar to the techniques employed to solve the inverse volcano problem for $N=1$ and $m=0$ in \cite{pazuki}.

Finally, we mention that several works on isogeny graphs with level structures have been released in recent years; see \cite{LM2,arpin,gorenkassaei,thesis-roda,codogni-lido,XZQ}. In a different vein, Pengo--Vallières \cite{PV} developed a general theory of graph coverings indexed by natural numbers of a given finite graph using Mahler measures. More specifically, given a finite graph $X$, they study a collection of graph coverings $\{X_n\}_{n\ge1}$ of $X$, where $X_n/X$ is a Galois covering whose Galois group is isomorphic to $\ZZ/n\ZZ$. It seems natural to study how isogeny graphs with level structures might fit in this framework.

\subsection*{Acknowledgement}
We thank Pete Clark, Daniel Larsson, Fabien Pazuki and Daniel Vallières for interesting discussions during the preparation of this article. We are also indebted to the anonymous referees for helpful comments and suggestions which led to many improvements in the article. The authors' research is supported by the NSERC Discovery Grants Program RGPIN-2020-04259 and RGPAS-2020-00096. Parts of this paper is based upon work supported by the National Science Foundation under Grant No. DMS-1928930 while A. L. was in residence at the MSRI / Simons Laufer Mathematical Sciences Institute (SLMath) in Berkeley, California, during the Spring 2023 semester.

\section{The definition of isogeny graphs and basic properties}\label{S:basic}

Throughout this article, a graph $X$ may be  directed or undirected. We write $V(X)$ for the set of vertices of $X$ and $\EE(X)$ for the set of edges of $X$. We shall say that $v\in V(X)$ admits an edge in $X$ if there exists $e\in \EE(X)$ such that $v$ is one of the end-points of $e$.

\subsection{Defining ordinary isogeny graphs}\label{S:define}
Let $p$ be an odd prime. We fix a second prime number $l\neq p$ and $N\ge1$ an integer coprime to $pl$. Further, we fix  a finite field $k$ of characteristic $p$.  We fix once and for all a set of  representatives of $\overline k$-isomorphism classes of elliptic curves defined over $k$, which we denote by $\mathcal{E}$. Note that a $\overline{k}$-isomorphism of  two curves over $k$ can be realized over the unique quadratic extension $k'$ of $k$ as long as the $j$-invariant is different from $0$ and $1728$.

We introduce the main object of interest of the present article:

\begin{defn}
\item[i)] Let $E$ and $E'$ be  elliptic curves defined over $k$. Let $P\in E(\overline{k})$ and $P'\in E'(\overline{k})$ be points of order $Np^m$. We say that  $(E,P)$ and $(E',P')$ are equivalent if there is a $\overline k$-isomorphism of elliptic curves $\phi:E\rightarrow E'$ with $\phi(P)=P'$.
\item[ii)] Let $E$ and $E'$ be  elliptic curves defined over $k$. Let $\phi$ and $\phi'$ be isogenies from  $E$ to $E'$ defined over $\overline k$. We say that $\phi$ and $\phi'$ are equivalent if $\ker(\phi)=\ker(\phi')$.
\item[iii)]For an integer $m\ge1$, we define a directed graph $G^m_N$ whose vertices are the equivalence classes of tuples $(E,P)$ given by i). There is a directed edge from $(E,P)$ to $(E',P')$ if and only if there is an  $l$-isogeny $\phi:E\rightarrow E'$  such that $\phi(P)=P'$, with each equivalence class of such isogenies gives rise to exactly one edge.
\end{defn}

By an abuse of notation, if $v\in V(G_N^m)$, we shall take any one representative of the equivalence class and simply write $(E,P)$. If $\phi$ gives rise to an edge from $(E,P)$ to $(E',P')$, we shall write $\phi(E,P)$, $(E',P')$ and $(\phi(E),\phi(P))$ interchangeably.

\begin{remark} Let $E\in\cE$ and $\alpha\in \textup{Aut}(E)$.
    By definition,  $(E,\alpha P)$ and $(E,P)$ give rise to the same vertex in $G_N^m$. 
    
    In the cases where $j(E)=0$ or $j(E)=1728$,  the group $\Aut(E)$ is strictly larger than $\{\pm 1\}$. Assume that this is the case and let $E'$ be an elliptic curve with $\textup{Aut}(E')=\{\pm1 \}$ such that there is an $l$-isogeny $\phi\colon E\to E'$. Let $\alpha\in \textup{Aut}(E)\setminus\{\pm1\}$. Then $(E,P)$ and $(E,\alpha P)$ define the same vertex in $G_N^m$, whereas $(E',\phi(P))$ and $(E',\phi(\alpha P))$  give rise to two distinct vertices. The isogenies $\phi$ and $\phi\circ \alpha$ are not equivalent, resulting in two edges from $(E,P)=(E,\alpha P)$ to $(E',\phi(P))$ and $(E',\phi(\alpha P))$, respectively.
    \e
\end{remark}

\begin{remark}
      If $E$ is an ordinary elliptic curve over $\overline{\Fp}$, there exists an imaginary quadratic field $K$ and an order $\mathcal{O}$ in $K$ such that $\textup{End}(E)=\mathcal{O}$. Note that $p$ is split in $K$. Thus, if $E$ is defined over $k$, then all of its endomorphisms are defined over $k$.\e
\end{remark}

\begin{remark} \label{rk:dual-isogeny}
    Let $(E,P)\in V(G_N^m)$. If $\phi\colon E\to E'$ is an $l$-isogeny that maps $P$ to $P'$, then the dual isogeny maps $P'$ to the point $lP$ on $lE$, where $lE$ denotes the image of the multiplication-by-$l$ map on $E$. The curve $lE$ is isomorphic to $E$. Let $\alpha$ be such an isomorphism. Then $(lE,lP)$ is equivalent  to $(E,\alpha(lP))$, which we explain below.
    
    Let  $\End(E)=\cO$ and $K=\End(E)\otimes\QQ$. By Deuring's lifting theorem (see \cite[Chapter 13, Theorem 14]{lang-elliptic}), there exists a lift  $\mathbf{E}$ of $E$ over a finite extension $L$ of $K$ and a prime ideal $\mathfrak{p}$ above $p$ such that $\mathbf{E}\pmod{\mathfrak{p}'}=E$ for some ideal $\mathfrak{p}'$ above $p$ in the ring of integers of $L$ and that $\End(\bE)=\End(E)=\cO$. Then $P$ admits a unique lift $\mathbf{P}\in \mathbf{E}[N\overline{\mathfrak{p}}^m]\cong E[Np^m]$, where $\overline{\mathfrak{p}}$ is an ideal above $p$ in $\mathcal{O}$ that is coprime to ${\mathfrak{p}'}$. (The isomorphism follows from the theory of formal groups when $\mathcal{O}=\mathcal{O}_K$; see for example \cite[Chapters I and II]{shalit}. As every CM elliptic curve is isogenous to one with complex multiplication by $\mathcal{O}_K$, this result easily carries over to general CM elliptic curves.)
    
    To determine $\alpha(lP)$, it suffices to consider $\tilde\alpha(l\mathbf{P})$, where $\tilde\alpha:l\mathbf{E}\rightarrow \mathbf{E}$ is a lift of $\alpha$. Furthermore, we may even base change to $\mathbb{C}$.
    We have $\mathbf{E}(\CC)=\CC/\Lambda$ for some lattice $\Lambda$ in $\CC$. Then, we may realize $l \bE(\CC) $ as $\CC/\frac{1}{l}\Lambda$ and 
    \begin{align*}
        \tilde\alpha:l \bE(\CC)&\rightarrow\bE(\CC)\\
        x+\frac{1}{l}\Lambda&\mapsto lx+\Lambda.
    \end{align*}
    Furthermore, we have
        \begin{align*}
        [l]: \bE(\CC)&\rightarrow l\bE(\CC)\\
        x+\Lambda&\mapsto x+\frac{1}{l}\Lambda.
    \end{align*}
    If we write $\mathbf{P}=x+\Lambda$, then $\tilde\alpha(l\mathbf{P})=l\mathbf{P}$. Thus $\alpha((lE,lP))=(E,lP)$ as claimed.\e
\end{remark}
 It follows from Remark~\ref{rk:dual-isogeny} that if $\phi$ is an isogeny that induces an edge from $(E,P)$ to $(E',P')$ in $G_N^m$, then the dual isogeny $\hat\phi$ gives an edge from $(E',P')$ to $(E,lP)$.

\begin{remark}
Each curve $E_{/\overline{\Fp}}$ admits $l+1$  isogenies that are of degree $l$. Let $\phi_1,\dots, \phi_{l+1}$ denote these isogenies. Note that  $\phi_i(E)$ may not be defined over $k$. As the vertices in $G_N^m$ are defined by elliptic curves over $k$, not all $\phi_i$ will necessarily result in an edge in $G_N^m$.
In general, the number of edges between $E$ and $E'$ is exactly the multiplicity of $j(E')$ as a root of the modular polynomial $\Phi_l(j(E),Y)$. \e
\end{remark}

If there is an $l$-isogeny between two elliptic curves $E$ and $E'$, then $\textup{End}(E)\otimes \Q=\textup{End}(E')\otimes \Q$. Thus,  $\End(E)$ and $\End(E')$ are orders in the same imaginary quadratic field. In particular, to each connected component of $G_N^m$, we may attach a unique imaginary quadratic field. This allows us to give the following definition:

\begin{defn}
Let $\cG_N^m$ be a connected component of $G_N^m$. Let $(E,P)$ be any vertex in $\cG_N^m$. We call $\End(E)\otimes \Q$ the \textbf{CM field} of $\cG_N^m$.
\end{defn}


\subsection{Coverings of ordinary isogeny graphs}
The goal of this section is to show that as $N$ and $m$ vary, the graphs introduced in the previous section give rise to coverings of graphs.  
Let $r,m,N,N'$ be nonnegative integers such that $N'|N$. There is a natural projection
\[\pi_{Np^{m+r}/N'p^m}\colon V(G_N^{m+r})\to V(G_{N'}^m)\]
given by $(E,P)\mapsto (E,\frac{Np^r}{N'}P)$. Further, if $\phi$ is an $l$-isogeny such that $\phi(E,P)=(E',P')$, we have $\phi\circ\pi_{Np^{m+r}/N'p^m}(E,P)=\pi_{Np^{m+r}/N'p^m}(E',P')$. Therefore, $\pi_{Np^{m+r}/N'p^m}$ extends to a map on the whole graph. By an abuse of notation, we shall write 
\[\pi_{Np^{m+r}/N'p^m}\colon G_N^{m+r}\to G_{N'}^m\]
for this map.

Let $\mathcal{G}_N^{m+r}$ be a fixed connected component of $G_N^{m+r}$ and let $\mathcal{G}_{N'}^m$ be the unique connected component of $G_{N'}^m$ that contains $\pi_{Np^{m+r}/N'p^m}(\mathcal{G}_N^{m+r})$.
\begin{lemma}\label{lem:cover}
The map $\pi_{Np^{m+r}/N'p^m}$ induces a covering of connected graphs $\mathcal{G}_N^{m+r}/\mathcal{G}_{N'}^m$.
\end{lemma}
\begin{proof}
    To simplify notation,  we write $\mathcal{G}$ for $\mathcal{G}_N^{m+r}$ and $\mathcal{H}$ for $\mathcal{G}_{N'}^m$. Further, we write $\pi$ for $\pi_{Np^{m+r}/N'p^m}$. We first show that  $\pi(V(\mathcal{G}))= V(\mathcal{H})$.

     Let $u=(E,P)\in V(\mathcal{G})$ and  let $v=(E', Q)\in V(\mathcal{H})$. Our goal is to find a preimage of $v$ in $V(\cG)$ under $\pi$. Let $w=(E,P_0)=\pi((E,P))$. As $\mathcal{H}$ is connected, there exists a path $C$ from $w$ to $v$. Let $n$ be the length of $C$ and write $\phi_1,\phi_2,\ldots \phi_n$ for the isogenies corresponding the edges of $C$. The composition $\Phi=\phi_n\circ\cdots \circ\phi_1$ is an isogeny from $E$  to $E'$ with $\Phi(\frac{Np^{r}}{N'}P)=Q$.
     
     Let $v'=(E',\Phi(P))\in G_N^{m+r}$. The isogenies $\phi_i's$ induce a path from $u$  to $v'$. Therefore, $v'$ belongs to $V(\cG)$. Furthermore, it is clear from definition that $\pi(v')=v$. This proves our claim above.

     As for edges, there is a one-one correspondence between the edges for which $(E,P)\in V(\cG)$ is the target (resp. source)  and the $l$-isogenies for which $E$ is the codomain (resp. domain). The same can be said for $\pi(E,P)=(E,\frac{Np^r}{N'}P)$. Thus, it follows that $\pi$ is locally an isomorphism of graphs, which concludes the proof of the lemma.
\end{proof}

Suppose that $N=N'$ and $r=1$. Since $E$ is defined over $\overline{\Fp}$ and $E$ is ordinary, the cover  $\cG_N^{m+1}\rightarrow \cG_N^m$ is of degree $p$. We will be interested in the tower of covers of the form
\[
\cG_N^m\leftarrow\cG_N^{m+1}\leftarrow\cG_N^{m+2}\leftarrow\cdots.
\]

\subsection{Horizontal and vertical edges}

In this section, we are interested in counting edges in and out of a vertex in $G_N^m$.
  A large part of our discussion therein has been inspired by the work of Bambury--Campagna--Pazuki \cite{pazuki}.

\begin{defn}\label{def:directions}
    Suppose that $\phi:E\rightarrow E'$ is an $l$-isogeny, and write $\mathcal{O}=\End(E)$ and $\mathcal{O'}=\End(E')$. Recall that $\mathcal{O}$ and $\mathcal{O}'$ are rings in the same imaginary quadratic field. There are three cases that may arise:
    \item[1)] $[\mathcal{O}:\mathcal{O}']=1$. In this case, we say that  $\phi$ is \textit{\textbf{horizontal}}.
    \item[2)] $[\mathcal{O}:\mathcal{O}']=l$. In this case, we say that $\phi$ is \textit{\textbf{descending}}.
    \item[3)] $[\mathcal{O}':\mathcal{O}]={l}$. In this case, we say that $\phi$ is \textit{\textbf{ascending}}. 
    We say that $\phi$ is \textbf{vertical} if it is either descending or ascending.

    If $\phi$ gives rise to $e\in \EE(G_N^m)$, we shall use the same terminology introduced above to describe $e$.
\end{defn}

\begin{defn}\label{def:level}
    Let $(E,P)$ be a vertex in $G_N^m$ and write $\End(E)=\Z+f\mathcal{O}_K$, where $f\in\ZZ$. 
    \item[i)] We call $v_l(f)$ the \textbf{level} of $E$. Note that this is well-defined (i.e., $v_l(f)$ is independent of the choice of $f$) and it only depends on the curve $E$,  not on the point $P$. 
    \item[ii)] The subgraph of $G_N^m$ generated by the set of vertices of level zero is called the \textbf{crater} (i.e., the maximal subgraph of $G_N^m$ containing all the level {zero} vertices). We write $\cC(G_N^m)$ for the crater of $G_N^m$.
\item[iii)]Let $\mathcal{G}_N^m$ be a connected component of $G_N^m$. We define the \textbf{depth} of $\mathcal{G}_N^m$ to be the maximal integer $d$ such that there is a vertex of level $d$ in $\mathcal{G}_N^m$. 
\item[iv)]If  $\mathcal{G}_N^m$ is as above, the crater of $\mathcal{G}_N^m$ is defined to be the intersection of $\cC(G_N^m)$ and $\mathcal{G}_N^m$. It will be denoted by $\cC(\cG_N^m)$.
 \end{defn}

Suppose that $N=1$ and $m=0$. The vertices of $G_1^0$ consist of isomorphism classes of elliptic curves (with $P$ taken as the identity element). Furthermore, if $\phi$ induces an edge $e$ from $E$ to $E'$, then the dual isogeny $\hat\phi$ of $\phi$ induces an edge $\hat e$ from $E'$ to $E$. Let $\mathfrak G$ denote the \textit{undirected} graph whose set of vertices is given by $V(G_1^0)$ and the set of edges is given by those in $\EE(G_1^0)/\sim$, where $\sim$ is the equivalence relation identifying $e$ with $\hat{e}$. The concepts introduced in Definitions~\ref{def:directions} and~\ref{def:level} carry over to $\fG$ naturally. The structure of $\fG$ can be described explicitly as follows.

 \begin{lemma}
 \label{lem:number-of-isogenies}
     Let {$\mathcal{G}$ be a connected component of $\fG$}. Suppose that it is of depth $d$ with CM field $K$ and that it contains no vertex $E$ with $j(E)\in\{0,1728\}$. Then the following statements hold.
          \item[i)] Let $v$ be a vertex of $\cC(\cG)$. It admits $1+\left(\frac{D_K}{l}\right)$ horizontal edges and no ascending edges. If $d>0$, it admits $l-\left(\frac{D_K}{l}\right)$ descending edges. If $d=0$, it admits no descending edge.
        \item[ii)] Let $1\le n\le d$ and let $v$ be a vertex in $\cG$ of level $n$. Then $v$ admits one ascending edge. If ${n<d}$, then $v$ admits $l-1$ descending edges. If $d=n$, then $v$ admits no descending edges.
    \end{lemma}
 \begin{proof}
     This follows directly from \cite[Proposition 3.17]{pazuki}.
 \end{proof}

This in turn allows us to describe the structure of $G_N^m$ when $Np^m$ is sufficiently large.
 
\begin{lemma} 
\label{lem:degree-split case}
Let $\mathcal{G}_N^m$ be a connected component of $G_N^m$. Let $d$ be the depth of $\mathcal{G}_N^m$. Assume that $\mathcal{G}_N^m$ does not contain a vertex of the form $(E,P)$ with $j(E)\in\{0,1728\}$. Let $K$ be the CM field of $\mathcal{G}_N^m$.
\item[i)]If $l$ splits in $K$ and $d=0$, each vertex in $\cG_N^m$ admits $4$ edges for $N$ or $m$ sufficiently large.
\item[ii)]Suppose that $d>0$ and that $l$ splits in $K$.  
    If  $N$ or $m$ is sufficiently large, each vertex in $\cC(\cG_N^m)$ admits $l+3$ edges in $\cG_N^m$.  For $1\le n\le d-1$, each vertex of level $n$ admits $l+1$ edges. Each vertex of level $d$  admits $1$ edge.
\item[iii)] Suppose that $l$ ramifies in $K$. If $d=0$, each vertex in $\cG_N^m$ admits $2$ edges for $N$ or $m$ sufficiently large.
\item[iv)] Suppose that $l$ ramifies in $K$ and that $d>0$. Each vertex in $\cC(\cG_N^m)$ admits $l+2$ edges in $\cG_N^m$ for $N$ or $m$ large enough. 
\item[v)] Suppose that $l$ is inert in $K$. If $d=0$, then each vertex in $\cC(\cG_N^m)$ is isolated. If $d>0$, then each vertex in $\cC(\cG_N^m)$ admits $l+1$ edges in $\cG_N^m$.
\end{lemma}
\begin{proof}
    Assume that $v=(E,P)$ is a vertex in $\cC(\cG_N^m)$. Let us first assume that $l$ splits in $K$, Lemma~\ref{lem:number-of-isogenies}i) tells us that  $v$ admits two horizontal isogenies $\phi_1$ and $\phi_2$ connecting $(E,P)$ to $(E',P')$ and $(E'',P'')$, respectively. The dual isogenies $\hat\phi_1$ and $\hat\phi_2$ give rise to edges going from $(E',\frac{1}{l}P')$ and $(E'',\frac{1}{l}P'')$ to $(E,P)$ (note that $l\nmid Np$ implies that $\frac{1}{l}P'$ and $\frac{1}{l}P''$ are well-defined).  For $m$ or $N$ large enough, the four vertices $(E',P')$, $(E'',P'')$, $(E',\frac{1}{l}P')$ and $(E'',\frac{1}{l}P'')$ are pairwise distinct. This proves part i) of the lemma.

    Now suppose that $d>0$. By Lemma~\ref{lem:number-of-isogenies}ii) $v$  admits $l-1$ descending edges and no ascending edges. Together with the 4 edges arising from the $4$ horizontal isogenies, we deduce that  $v$ admits $(l-1)+4=l-3$ edges.

   The other cases can be proved similarly.
\end{proof}

We consider the cases where $j(E)\in\{0,1728\}$ separately.

\begin{lemma}
    Let $E$ be an elliptic curve with $j(E)=0$. Let $\mathcal{G}_N^m$ be a connected component of $G_N^m$ containing a vertex of the form $v=(E,P)$.
    \item[i)] If $l=3$, then each vertex in $\cC(\cG_N^m)$ admits $5$ edges in $\cG_N^m$ for $N$ or $m$ sufficiently large.
    \item[ii)] If $l\equiv 1\pmod 3$ and $d=0$, then each vertex in $\cC(\cG_N^m)$ admits $4$ edges in $\cG_N^m$ for $N$ or $m$ sufficiently large.
    \item[iii)] If $l\equiv 1\pmod 3$ and $d>0$, then each vertex in $\cC(\cG_N^m)$ admits $l+3$ edges in $\cG_N^m$.
    \item[iv)] If $l\equiv 2\pmod 3$ and $d=0$, then each vertex in $\cC(\cG_N^m)$ is isolated.
    \item [v)] If $l\equiv 2\pmod 3$ and $d>0$, then each vertex in $\cC(\cG_N^m)$ admits $l+1$ edges in $\cG_N^m$.
\end{lemma}
\begin{proof}
    The proof is essentially the same proof as that Lemma~\ref{lem:degree-split case} upon replacing Lemma~\ref{lem:number-of-isogenies} by \cite[Proposition 3.19]{pazuki}.
\end{proof}
\begin{lemma}
     Let $E$ be an elliptic curve with $j(E)=1728$. Let $\mathcal{G}_N^m$ be a connected component of $G_N^m$ containing a vertex of the form $v=(E,P)$.
    \item[i)] If $l=2$, then each vertex in $\cC(\cG_N^m)$ admits $4$ edges in $\cG_N^m$ for $N$ or $m$ sufficiently large.
    \item[ii)] If $l\equiv 1\pmod 4$ and $d=0$, then each vertex in $\cC(\cG_N^m)$ admits $4$ edges in $\cG_N^m$ for $N$ or $m$ sufficiently large.
    \item[iii)] If $l\equiv 1\pmod 4$ and $d>0$, then each vertex in $\cC(\cG_N^m)$ admits $l+3$ edges in $\cG_N^m$.
    \item[iv)] If $l\equiv 2\pmod 4$ and $d=0$, then each vertex in $\cC(\cG_N^m)$ is isolated.
    \item [v)] If $l\equiv 2\pmod 4$ and $d>0$, then each vertex in $\cC(\cG_N^m)$ admits $l+1$ edges in $\cG_N^m$.
\end{lemma}
\begin{proof}
    This follows again from the same proof as that of Lemma~\ref{lem:degree-split case} upon employing \cite[Proposition 3.20]{pazuki}.
\end{proof}

\section{Abelian $p$-towers}
The main goal of this  section is to prove Theorem~\ref{thmA}. Throughout, the notation introduced in \S\ref{S:basic} continues to be in force. We begin by recalling the  definition of an abelian $p$-tower from \cite[Definition~4.1]{vallieres}:

\begin{defn}\label{def:tower}
An abelian $p$-tower of undirected graphs above a graph $X$ is a sequence of covers
\[
X=X_0\leftarrow X_1\leftarrow X_2\leftarrow \cdots \leftarrow X_n\leftarrow\cdots
\]
such that for each $n\ge0$, the cover $X_n/X$ is abelian with Galois group isomorphic to $\ZZ/p^n\ZZ$.
\end{defn}
We are interested in the case where $X_n$ are connected for all $n$. If $X_0$ admits more than one connected component, the Galois group of $X_n/X$ may potentially be a direct product of two or more groups, which we would like to avoid. 

\begin{defn}
The  undirected graph obtained  by ignoring directions of  the edges in $G_N^m$ is denoted by $\tilde G_N^m$. Similarly if $\cG_N^m$ is a connected component of $G_N^m$, we define $\tilde \cG_N^m$ similarly.
\end{defn}

We shall take $X$ to be a single connected component $\tilde \cG_N^{m_0}$, where $m_0$ is a sufficiently large integer that will be given in  \S\ref{S:connectedcomponents} below.

\subsection{Connected components}\label{S:connectedcomponents}

The goal of this section is to study the number of connected components of $G_N^m$ as $m$ varies. 

\begin{proposition}\label{prop:components}
Assume that none of the connected components of $G_1^0$ is a single vertex without any edges. 
There exists an integer $m_0$ such that the number of connected components in $G_N^m$ is the same  as $G_N^{m_0}$ for all $m\ge m_0$.     \end{proposition}
\begin{proof}
    Let $s_N^m$ be the number of connected components in $G_N^m$. It follows from Lemma~\ref{lem:cover} that   $\pi_{Np^{m+1}/Np^m}:G_N^{m+1}\rightarrow G_N^m$ is a graph covering. Thus, $s_N^{m+1}\ge s_N^m$. It remains to show that this sequence stabilizes when $m$ is sufficiently large.

    Suppose that $m\ge1$. Since $E$ is ordinary at $p$ and $p\nmid N$, we have the group isomorphisms
    \[
    \Aut(E[p^mN])\cong (\cO/N\cO)^\times\times (\ZZ/p^{m}\ZZ)^\times,
    \]    where $\cO$ is an order in an imaginary quadratic field.
    Therefore, the order of $[l]$ (the multiplication by $l$ map) in $\textup{Aut}(E[p^mN])$ is equal to the multiplicative order of $l$ in the group $(\ZZ/Np^m\ZZ)^\times$. Thus, there exist integers  $m_0$ and $c$ such that this order is given by $cp^{m-m_0}$ for all $m\ge m_0$. For all such $m$, we have
    \[l^{cp^{m-m_0}}\equiv 1\pmod {Np^m}.\]
    In particular, $l^c\equiv 1\pmod {Np^{m_0}}$. It follows that 
    \begin{equation}
        \label{eq:kill}
    l^cp^{m-m_0}\equiv p^{m-m_0}\pmod {Np^{m}}.
       \end{equation} 
  
    Let $v_0=(E,P_0)\in V(G_N^{m_0})$ and let $v=(E,P)\in V(G_N^m)$ be a pre-image of $v_0$ under $\pi_{Np^m/Np^{m_0}}$, where $m\ge m_0$.  Consider the set 
    \[
     C:=\left\{(E,l^{cn}P):n\in \ZZ_{\ge0}\right\}.
    \]
      By \eqref{eq:kill}, all elements of $C$ are sent to $v_0$ under $\pi_{Np^m/Np^{m_0}}$. Furthermore, the fact that  the order of $[l]$ in $\textup{Aut}(E[p^mN])$ equals $cp^{m-m_0}$ implies that  $C$ contains exactly the $p^{m-m_0}$ elements. In particular, $C$ is precisely the pre-image of $v_0$ in $G_N^m$. By our assumption on $G_1^0$, $E$ admits an $l$-isogeny. Thus, as we have seen in Remark~\ref{rk:dual-isogeny}, all vertices in $C$  lie in the same connected component of $G_N^{m}$ as $v$. This implies that the number of connected components stabilizes and concludes the proof.
\end{proof}

\begin{remark}\label{rk:components}
    Assume that there is an isolated vertex $E$ in $G_N^1$. Then for all $m\ge1$ and all $P\in E[Np^m]$,  the vertex $(E,P)\in G_N^m$ is also isolated. In particular, as we pass from $G_N^m$ to $G_N^{m+1}$, the number of connected components arising from $E$  is multiplied by $p$. Therefore, the number of connected components in $G_N^m$ is unbounded as $m\rightarrow\infty$.

    Let $\mathcal{E}'$ be the set of $\overline{k}$ equivalence classes of ordinary elliptic curves over $k$ that are not isolated in $G_1^0$, i.e. $\mathcal{E}'$ only contains isomorphism classes of curves admitting a degree $l$ isogeny to a curve defined over $k$. Let $H_N^m\subset G_N^m$ be the subgraph on the vertices of the form $(E,P)$ with $E\in\mathcal{E}'$. Then the proof of Proposition~\ref{prop:components} shows that the number of connected components of $H_N^m$ is constant for $m\ge m_0$.
    
    In particular, if $m\ge m_0$ and $\cG^m_N$ is a connected component of $H_N^m$, then $H_N^{m+r}$ admits a unique connected component whose image under  $\pi_{Np^{m+r}/Np^m}$ is $\cG_N^m$.    
    \e
\end{remark}

\begin{defn}
    Given an integer $N\ge1$, we write $m_0$ to be the integer given by Remark~\ref{rk:components}.
\end{defn}

\subsection{Galois covers and abelian $p$-towers}

We now prove a proposition regarding the cover $G_N^{m+r}/G_N^m$, which will imply  Theorem~\ref{thmA} stated in the introduction. We shall work with undirected graphs, following works on Iwasawa theory of graphs in the literature, in particular \cite{vallieres,vallieres2,vallieres3,DLRV}.

\begin{proposition}\label{prop:Galois-cover}
    Let $H_N^m$ and $m_0$ be defined as in Remark~\ref{rk:components}.
     Let $\mathcal{G}_N^m$ be a connected component of $H_N^m$ and fix $m\ge \textup{max}(m_0,1)$. Let $\mathcal{G}_N^{m+r}$ be the connected component of $H_N^{m+r}$ that maps onto $\mathcal{G}_N^m$ via $\pi_{Np^{m+r}/Np^m}$. Then $\tilde\cG_N^{m+r}/\tilde\cG_N^m$ is a Galois graph covering whose Galois group is isomorphic to $\ZZ/p^r\ZZ$.   
\end{proposition}
 
\begin{proof}
    Let $U\subset (\Z/Np^{m+r}\ZZ)^\times$ be the subgroup consisting of elements that are congruent to $1$ modulo $Np^m$. Note that $U\cong \ZZ/p^r\ZZ$ as abelian groups.

    Let $\pi=\pi_{Np^{m+r}/Np^m}$. We define an action of $U$ on $V({G}_N^{m+r})$ by 
    \[a\cdot (E,P)=(E,aP).\]
    As $\pi(E,P)=\pi(E,aP)$. It follows from the proof of Proposition~\ref{prop:components} that $(E,aP)$ and $(E,P)$ lie in the same connected component of $G_N^{m+r}$. In particular,  the action defined above restricts  to an action of $U$ on  $V(\mathcal{G}_N^{m+r})$. 

    This action extends to a graph homomorphism of $\tilde \cG_N^{m+r}$. Indeed, let $(E,P)$ and $(E',P')$ be adjacent vertices in $\tilde\cG_N^{m+r}$, connected by an edge $e$. Without loss of generality, we can assume that $e$ is induced by an  $l$-isogeny $\phi\colon E\to E'$ such that $\phi(P)=P'$.  Then the same isogeny induces an edge between $(E,aP)$ and $(E',aP')$ since $\phi(aP)=a\phi(P)=aP'$. 

    Let $\textup{Deck}(\tilde\cG_N^{m+r}/\tilde\cG_N^m)$ denote the group of deck transformations of the graph covering  $\pi:\tilde\cG_N^{m+r}\rightarrow\tilde\cG_N^m$, whose degree equals $p^r$. Recall that $\textup{Aut}(E)\in\{\{\pm 1\},\mu_6,\mu_4\}$. Let $K$ be the CM field of $\mathcal{G}_N^m$. Let $\mathfrak{p}$ be a prime above $p$ in $K$. Then $\alpha\not \equiv 1 \pmod {\mathfrak{p}^2}$ for every $\alpha\in \textup{Aut}(E)$ {that is not the identity}. In particular,  $(E,P)$ and $(E,aP)$ are two distinct vertices. Thus, the action of $U$ on $\tilde\cG_N^{m+r}$ induces  an injective group homomorphism
    \[U\hookrightarrow \textup{Deck}(\tilde\cG_N^{m+r}/\tilde\cG_N^m).\]
    To show that $\pi$ is a Galois cover whose Galois group is isomorphic to $\ZZ/p^r\ZZ$, it remains to show that this injective group homomorphism is  surjective.
    
    Let $\psi$ be a deck transformation and let $(E,P)\in V(\tilde\cG_N^{m+r})$. We write $\psi(P)$ for the point of order $p^{m+r}N$ such that $\psi(E,P)=(E,\psi(P))$. As $\pi(E,P)=\pi(\psi((E,P)))$, we have 
    $$ P-\psi(P)\in E[p^r].$$
    In particular, $\psi(P)=a P$ for some $a\in U$. Therefore, $\psi((E,P))=a\cdot (E,P)$.
    
    It remains to show that 
    \begin{equation}
     \psi(E',P')=a\cdot (E',P')\label{eq:toshowsurj}     
    \end{equation}
for all  $(E',P')\in V(\tilde\cG_N^{m+r})$. Let us first consider the case $ (E',P')$ is a vertex that is adjacent to $(E,P)$. Suppose that there is a degree $l$ isogeny $\phi\colon E\to E'$ with $\phi(P)=P'$. As $\psi$ is a deck transformation, we have the following commutative diagram
    \[\xymatrix{
    (E,P)\ar[r]^\phi\ar[d]^\psi &(E',P')\ar[d]^\psi\\
    (E,\psi(P))\ar[r]^\phi& (E',\psi(P')).
    }\]
    Therefore,
    \[\psi(E',P')=(E', \phi(\psi(P)))=(E', \phi(aP))=(E', a \phi(P))=(E',a P')=a(E',P').\]
    
    Assume now that there is an isogeny $\phi \colon E'\to E$ of degree $l$. Then
    \[\phi(aP')=aP=\psi(P)=\phi(\psi(P')).\]
    As $\phi$ is injective on $E'[Np^{m+r}]$ it follows that $aP'=\psi(P')$. So in both cases we have shown that $\psi((E',P'))=a(E',P')$.
    
    As $\tilde\cG_N^{m+r}$ is connected, we deduce that \eqref{eq:toshowsurj} holds for all $(E',P')$ as required.
\end{proof}
Proposition~\ref{prop:Galois-cover} implies immediately Theorem~\ref{thmA} stated in the introduction:
\begin{corollary}\label{cor:main}
    The  graph coverings
    \[
    \tilde\cG_N^{m_0}\leftarrow\tilde\cG_N^{m_0+1}\leftarrow\cdots
    \]
    is an abelian $p$-tower in the sense of Definition~\ref{def:tower}.
\end{corollary}

\section{The structure of the crater}
\label{S:crater}
The goal of this section is to study the connected components of the crater $\cC(G_N^m)$ for any given $N$ and $m$. The CM field of this chosen connected component will be denoted by $K$ throughout.

We consider two separate cases. Namely, when $l$ is non-split in $K$ and when $l$ is split in $K$. The split case turns out to be  more delicate than the non-split case.

\subsection{The non-split case}

We first study the case where $l$ is non-split in $K$. This can be divided further into two sub-cases, namely either $l$ is inert or ramified in $K$.

\begin{lemma}
    Let $\cC_0$ be a connected component of $\cC(G_N^m)$. 
    \item[i)] If $l$ is inert in $K$, then $\cC_0$ consists of a single vertex, without any edges.
    \item[ii)] If $l$ is ramified in $K$, then $\cC_0$ is either a single vertex with a loop or it is a directed cycle, i.e., {there exists an integer $s\ge1$ such that $\cC_0$} consists of $s$ vertices $\{v_1,\dots v_s\}$ with edges going from from $v_i$ to $v_{i+1}$ for $1\le i\le s-1$ and an edge from $v_s$ to $v_1$. 
\end{lemma}
\begin{proof}
If there is an edge between two vertices of level zero, it has to be induced by a horizontal isogeny.
    Part i) follows immediately from \cite[Corollary 3.13]{pazuki}. 
    
    We now prove part ii). Let $(E,P)\in V(\cC_0)$. We set $\cO=\End(E)$. Let $\fL$ be the ideal of $\cO$ above $l$ (i.e. $l\cO=\fL^2$). Note that $\fL^2$ is a principal ideal. 
    
    If $\fL$ itself is principal, then there exists an element $x\in \mathcal{O}$ such that $E/E[\fL]=xE\cong E$. Let $h$ be the order of $x$ in $(\mathcal{O}/N\mathfrak{p}^m)^\times/\mathcal{O}^\times$, where we have written $\cO^\times$ for its natural image in $(\mathcal{O}/N\mathfrak{p}^m)^\times$ and $\mathfrak{p}$ is an ideal of $\cO$ lying above $p$.\footnote{The choice of $\mathfrak{p}$ depends on the choice of a CM lift $\mathbf{E}$ of $E$ over some finite abelian extension $L/K$ together with a prime ideal $\mathfrak{p}'$ above $p$ in $\mathcal{O}_L$ such that $\mathbf{E}\pmod {\mathfrak{p}'}=E$. As the integer $h$ does not depend on the choice of $\mathfrak{p}$, we suppress this choice here.} If $h=1$, then $\cC_0$ consists of a single vertex $(E,P)$ together with a loop. If $h>1$, then $\cC_0$ is  a directed cycle of length $h$, with vertices given by $(E,x^iP)$, $i=0,2,\ldots, h-1$, and edges given by $(E,P)\rightarrow (E,xP)\rightarrow\cdots\rightarrow (E, x^{h-1}P)\rightarrow (E,P)$. 
    
    If $\fL$ is not principal, the curve $E':=E/E[\fL]$ is not isomorphic to $E$. There is an $l$-isogeny $\phi:E\rightarrow E'$ and a dual isogeny $\hat \phi:E'\to E'/E'[\fL]\cong E$. Let $h$ be the order of $l$ in $(\mathcal{O}/N\mathfrak{p}^m)^\times/\mathcal{O}^\times$. Then $\cC_0$ is a directed cycle of length $2h$, with vertices given by $(E,l^iP)$ and $(E',\phi(l^iP))$, $i=1,\ldots, h$; the edges are given by $(E,l^iP)\rightarrow (E',\phi(l^iP))$ and $(E',\phi(l^iP))\rightarrow(E,l^{i+1}P)$. 
 \end{proof}
\begin{remark}
    In the inert case, while $\cC_0$ consists of a single vertex, there may be edges in $G_N^m$ connecting it to other vertices in $\cC(G_N^m)$ via edges arising from vertical isogenies. 
    
    Let $v=(E,P)\in V(\cC_0)$ and assume that $v$ admits an edge in $G_N^m$. Let $\cG_N^m\subset G_N^m$ be the connected component containing $v$. Then all the vertices of the form $(E,l^tP)$, $t\in\ZZ$, lie in $\cG_N^m$, as we have seen in Remark~\ref{rk:dual-isogeny}. Suppose that $v'\neq v$ is any vertex in $\cC(\cG_N^m)$.  We claim that $v'$ is of the form $(E,l^tP)$. 
    
    Indeed, as $E$ admits no horizontal isogeny, a level zero vertex of $\mathcal{G}_N^m$ is of the form $(E,P')$ for some $P'$. Thus, there is an endomorphism of $E$ of $l$ power degree that maps $P$ to $P'$. As $l$ is inert, the only $l$ power degree endomorphisms are given by powers of $[l]$. Thus, $P'=l^t P$ for some $t$ as claimed. \e
\end{remark}
\begin{remark}
    If $v\in \cC_0$ and $\fL$ is ramified or split in $K$, then all level zero vertices that lie in the same connected component of  $G_N^m$ as $v$ are also elements of  $V(\cC_0)$. \e
\end{remark}
 \subsection{Classification of vertices and edges in the split case}\label{S:split-crater}
 From now on, we assume that $l$ splits in $K$. Let $v_1=(E,P)$ be a fixed level zero vertex in $G_N^m$ and let $\mathcal{O}=\textup{End}(E)$. The ideal $l\mathcal{O}$ splits into two distinct ideals $\mathfrak{L}$ and $\overline{\mathfrak{L}}$. We continue to write $\cC_0$ for the connected component of $\cC(G_N^m)$ containing $v_1$.
 
 Note that all level zero vertices connected to $v_1$ are connected through horizontal isogenies. Further, $E$ admits exactly two horizontal isogenies, namely $E\to E/E[\fL]$ and $E\to E/E[\overline{\fL}]$. Thus, all elements of $V(\cC_0)$ arise from the curves of the form $E/E[\fL^a\overline{\fL}^b]$.  In particular, all these curves have complex multiplication by the same ring $\mathcal{O}$. An edge between two such vertices arises from either $\fL$ or $\overline{\fL}$. 
 
 It follows from Lemma~\ref{lem:degree-split case} that when either $m$ or $N$ is sufficiently large, there are  two edges $\EE(\cC_0)$ with $v_1$ as the source and two edges with $v_1$ as the target. We see that $\fL$  induces precisely one of the former and one of the latter, whereas the other two edges are induced by $\overline{\fL}$.

 The following lemma studies equalities in $V(\cC_0)$.

\begin{lemma}
\label{lem:connected}
     Let $v_i=(E_i,P_i)\in V(\cC_0)$, $i=1,2$. Suppose that
      \[(E_1/E_1[\fL^a\overline{\fL}^b],P_1+E_1[\fL^a\overline{\fL}^b])=(E_1/E_1[\fL^d\overline{\fL}^e],P_1+E_1[\fL^d\overline{\fL}^e])\]
    as elements of $V(\cC_0)$ for some nonnegative integers $a,b,d$ and $e$. Then
    \[(E_2/E_2[\fL^a\overline{\fL}^b],P_2+E_2[\fL^a\overline{\fL}^b])=(E_2/E_2[\fL^d\overline{\fL}^e],P_2+E_2[\fL^d\overline{\fL}^e])\]
    as elements of $V(\cC_0)$.
 \end{lemma}
 \begin{proof}
 Let $i\in\{1,2\}$ and $\alpha,\beta$ be nonnegative integers. We write $\phi_{\alpha,\beta,i}:E_i\rightarrow E_i/E_i[\cL^\alpha\overline{\cL}^\beta]$ for the isogeny given by the natural projection.
 Let us write $v_3=\phi_{a,b,1}(v_1)=\phi_{d,e,1}(v_1)$. We define $\phi_{\alpha,\beta,3}$ similarly.

Since $v_1$ and $v_2$ are level zero vertices lying in the same connected component of $\cC(G_N^m)$, there is a path in $\cC_0$ connecting $v_1$ to $v_2$. Thus, upon propagating along this path,  we may assume that there is an edge in $\EE(\cC_0)$ connecting $v_1$ to $v_2$. In this case,  this edge is induced by $\phi_{1,0,1}$ or $\phi_{0,1,1}$. Thus it suffices to prove the lemma for these two cases.

     Suppose that $\phi_{1,0,1}$ induces an edge from $v_1$ to $v_2$. Then one can check directly from definition that
     $$ \phi_{a,b,2}(v_2)=\phi_{a+1,b,1}(v_1)=\phi_{1,0,3}\circ\phi_{a,b,1}(v_1)=\phi_{1,0,3}(v_3).$$
     Similarly,
     $$ \phi_{d,e,2}(v_2)=\phi_{d+1,e,1}(v_1)=\phi_{1,0,3}\circ\phi_{d,e,1}(v_1)=\phi_{1,0,3}(v_3).$$
This proves the desired equality. The other case can be treated in the same manner.
 \end{proof}


\begin{defn}\label{def:GNm}
\item[i)]We call an edge in  $\cC_0$   \textbf{blue} if it is induced by the isogeny given by $\mathfrak{L}$ and we call it \textbf{green} if it is induced by $\overline{\mathfrak{L}}$. 
\item[ii)] We call a path in $\cC_0$ \textbf{blue} if it only consists of blue edges, we call it \textbf{green} if it only consists of green edges. 

\item[iii)] Let $h_1$ (resp. $h_2$) be the minimal length of a closed blue (resp. green) path starting at $v_1$ in $\cC_0$ without backtracks. 
\end{defn}
Note that all edges in $\cC_0$ are either blue or green (but not both). If we repeatedly apply $\fL$ to $v_1$, we will eventually obtain $v_1$. Indeed, there exists a nonnegative integer $n$ such that $\fL^n$ is a principal ideal $\gamma\cO$, say. As $l\nmid Np$, there exists an integer $n'$ such that  $\gamma^{n'} P=P$. By a similar argument to the one presented in Remark~\ref{rk:dual-isogeny}, there is a blue path of length $nn'$  sending $v_1$ to itself. This tells us that $h_1$ is finite. The same holds for $h_2$.

If $(E',P')$ lies on  a blue (resp. green) path originating from $v_1$, then $E'= E/E[\fF^a]$, $P'=P+E[\fF^a]$ for some integer $a$ and $\fF=\fL$ (resp. $\overline\fL$). It can happen that these two classes of vertices coincide. We study this phenomenon in the following proposition.

\begin{proposition}
    \label{prop:st}
There are positive integers $s\mid h_1$ and $t\mid h_2$ and a positive integer $c$ coprime to $h_2/t$ such that \[(E/E[\fL^s],P + {E[\fL^s]})=(E/E[\overline{\fL}^{ct}],P+ {E[\overline{\fL}^{ct}]})\]
as elements of $V(\cC_0)$. Suppose that $(s,t,c)$ is such a tuple with $s$ minimal. If $(s',t',c')$ is another tuple such that \[(E/E[\fL^{s'}],P + {E[\fL^{s'}]})=(E/E[\overline{\fL}^{c't'}],P+ {E[\overline{\fL}^{c't'}]}),\] then there exists  $d\in \ZZ$ such that $s'=ds$ and $t'=dt$. In particular, the minimality of $s$ implies the minimality of $t$.
\end{proposition}
   
  \begin{proof}
      Let $\mathcal{S}$ be the set of all tuples $(s_1,t_1,c_1)$ such that $s_1\mid h_1$, $t_1\mid h_2$ and $$(E/E[\fL^{s_1}],P+{E[\fL^{s_1}]})=(E/E[\overline{\fL}^{c_1t_1}],P+{E[\overline{\fL}^{c_1t_1}]}).$$
 Note that $\mathcal{S}$ is non-empty since it contains $(h_1,h_2,1)$.
 
 Let  $(s,t,c)\in\cS$ such that $s$ is minimal. It follows from Lemma \ref{lem:connected} that $h_1/s=h_2/t$. It implies that $t$ is also minimal. Let $(s',t',c')\in\mathcal{S}$ and write $s''=\textup{gcd}(s',s)$. There are nonnegative integers $a_1,a_2$ such that $s''\equiv a_1s+a_2s'\pmod {h_1}$. Lemma \ref{lem:connected} implies that
\begin{align*}(E/E[\fL^{s''}],P+E[\fL^{s''}])&=(E/E[\fL^{a_1s+a_2s'}],P + {E[\fL^{a_1s+a_2s'}]})\\&=(E/E[\overline{\fL}^{a_1ct+a_2c't'}],P+{E[\overline{\fL}^{a_1ct+a_2c't'}]}).\end{align*}
As $s$ is minimal, we see that $s''=s$ and $s'=ds$ for some integer $d$. It follows that $t'=h_2 s'/h_1=d h_2s/h_1=dt$ as required. 
\end{proof}
From now on, we shall always write $s$ and $t$ for the minimal integers given by {Proposition}~\ref{prop:st}.

\begin{defn}
We call a vertex $v\in V(\cC_0)$ \textbf{central} if there exists some positive integer $\alpha$ and a blue path of length $\alpha s$ connecting $v_1$ to $v$.\end{defn}
\begin{remark}\label{rk:central-defn}
    Suppose that $v$ is a central vertex. Then Proposition~\ref{prop:st} tells us that there is a nonnegative integer $\alpha'$ such that there is a green path of length $\alpha't$ connecting $v_1$ to $v$. In other words, we may give an equivalent definition central vertices using green paths.
\end{remark}

\begin{defn}\label{def:primary}
    Let $v$ and $v'$ be any two central vertices connected through a blue path of length $s$. We call the $s-1$ vertices on this path that are different from $v$ and $v'$ \textbf{blue primary vertices}. We define \textbf{green primary vertices} similarly for a green  path of length $t$ between two central vertices.
\end{defn}

{\begin{remark}
The blue primary vertices introduced in Definition~\ref{def:primary} only exist if $s>1$. Similarly, the green primary ones only exist if $t>1$.\e
\end{remark}}
To illustrate, suppose that $h_1=12$ and $s=4$. Then we have a directed cycle $v_1\rightarrow v_2\rightarrow \cdots \rightarrow v_{12}\rightarrow v_1$ of length 12 passing through $v_1$, consisting of blue edges. The vertices $v_1,v_5,v_9$ are central, whereas the rest of the vertices on the cycle are blue primary. 
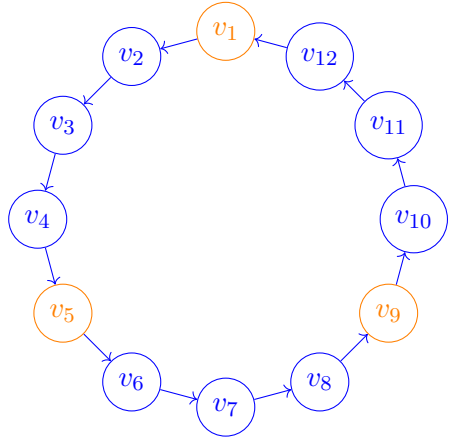
\begin{center}
\begin{tikzpicture}[vertex/.style={circle, draw}] 
\foreach \X[count=\Y] in {orange,blue,blue,blue,orange,blue,blue,blue,orange,blue,blue,blue}
{\node[draw,circle,\X] (x-\Y) at ({30*\Y+60}:2.5){$v_{\Y}$}; }
\foreach \X[count=\Y] in {0,...,11}
{\ifnum\X=0
\draw[->,blue] (x-12) --(x-\Y) ;
\else
\draw[->,blue] (x-\X) --(x-\Y) ;
\fi}
\end{tikzpicture}
\end{center}

We prove in the following lemma that central vertices and blue primary vertices are in fact mutually exclusive.

\begin{lemma}\label{lem:primary-not-central}
    The blue primary vertices and green primary vertices are not central.
\end{lemma}
\begin{proof}
    We only consider blue primary vertices; the other case can be proved in a similar manner. Let $v$ be a blue primary vertex. By definition, it lies on a blue path of length $s$ from one central vertex  to another, say $w$. In particular, there is a blue path of length $a$ going from $v$ to  $w$, where  $1\le a\le s-1$. 
    
Suppose that $v$ is central. Since both $w$ and $v$ are central, it follows from Proposition~\ref{prop:st} that there is a blue path of length $a's$ going from $w$ to $v$ for some nonnegative integer $a'$.   Consequently, we obtain a closed blue path from $v$ to itself of length $a+a's$. 

By Proposition~\ref{lem:connected}, there is a closed blue path of length $a+a's$ from $v_1$ to itself.  Proposition~\ref{prop:st} says that $s$ divides $a+a's$. But this contradicts that $1\le a\le s-1$. Thus, $v$ is not central.
\end{proof}

\begin{remark}\label{rk:central}
    Lemma~\ref{lem:primary-not-central} tells us that there are in total $h_1/s$ central vertices, equally distributing along a closed blue path of length $h_1$ passing through $v_1$. Furthermore, Remark~\ref{rk:central-defn} tells us that we may equally count $h_2/t$  central vertices on a closed green path of length $h_2$ passing through $v_1$. In particular, we have the equality
    \[
    \frac{h_1}{s}=\frac{h_2}{t}.
    \]\e
\end{remark}

\begin{lemma}
\label{number-blue}
 There are $\displaystyle\frac{h_1}{s}\cdot(s-1)$ blue primary vertices and $\displaystyle\frac{h_2}{t}\cdot (t-1)$ green primary vertices, respectively.
\end{lemma}

\begin{proof}We only prove the statement on blue primary vertices.
The closed blue path of length $h_1$ gives rise to $h_1/s$ blue paths of length $s$, each of which connecting two central vertices. Each of these paths in turn gives rise to  $s-1$ blue primary vertices. Thus,  $\displaystyle\frac{h_1}{s}\cdot (s-1)$ is  an upper bound on the total number of blue primary vertices. 

We deduce from the previous paragraph that the total number of blue and central vertices is bounded above by $$\frac{h_1}{s}+\frac{h_1}{s}\cdot(s-1)=h_1.$$ All these vertices are connected through blue edges and lie on a closed blue path (without backtracks) through  $v_1$. The minimal number of edges needed to draw a closed blue path is $h_1$.  Therefore, this upper bound is in fact optimal. Thus, there are exactly $\displaystyle\frac{h_1}{s}\cdot(s-1)$ blue primary vertices.
\end{proof}
\begin{lemma}
There is no vertex that is simultaneously blue primary and green primary.
\end{lemma}
\begin{proof}
    Suppose that $v$ is a vertex that is both blue primary and green primary. There exist central vertices  $w$ and $w'$ such that $w$ is connected to $v$ through a blue path of length $1\le a\le s-1$ and  $w'$ is connected to $v$ through a green path of length $1\le a'\le t-1$. 
    
    Since both  $w$ and $w'$ are central, they are connected through a green path of length $a''t$ for some nonnegative integer $a''$. It follows that the isogenies induced by $\mathfrak{L}^a$ and by $\overline{\mathfrak{L}}^{a'+ta''}$ coincide on $v$. But $0<a<s$ and $0<a'<t$, which contradicts Propoistion~\ref{prop:st}. Thus, such $v$ does not exist.
 \end{proof}

Suppose that both $s$ and $t$ are strictly greater than $1$. Let $v$ be a blue primary vertex. It follows from Lemma~\ref{lem:primary-not-central} that $v$ is not central. Consequently, there is a green path of length $t$ from $v$ to some blue primary vertex $v'$. We shall study the vertices appearing on such a path.

 \begin{defn} Excluding the end-points, we call the vertices lying on a green path linking two blue primary vertices \textbf{green secondary vertices}. We define \textbf{blue secondary} vertices in a similar manner.\end{defn}

By Lemma~\ref{number-blue}, there are $h_1-h_1/s$ blue primary vertices. A green path connecting two blue primary vertices has length $t$ (following from Lemma~\ref{lem:connected} and Proposition~\ref{prop:st}). Thus, the number of green secondary vertices is bounded above by
\[
\left(h_1-\frac{h_1}{s}\right)(t-1)=\frac{h_1}{s}(s-1)(t-1).
\]
Similarly, the number of blue secondary vertices is bounded above by
\[
\left(h_2-\frac{h_2}{t}\right)(s-1)=\frac{h_2}{t}(s-1)(t-1).
\]
These two upper bounds are equal to each other since $h_1/s=h_2/t$ by  Remark~\ref{rk:central}. 
 \begin{lemma}
     There are exactly $\displaystyle\frac{h_1}{s}(s-1)(t-1)=\frac{h_2}{t}(s-1)(t-1)$ blue/green secondary vertices.
 \end{lemma}
 \begin{proof}
 Take any two central vertices linked by a blue path of length $s$. This gives  $s-1$ blue primary vertices  lying on a blue path of length $s-2$.  Through each of these vertices, there exists a green cycle of minimal length, i.e., a cycle of length $h_2$ obtained by repeatedly applying $\overline{\fL}$. Let's call them $C_1,\ldots, C_{s-1}$. These cycles are disjoint by construction. 

Let $v=(E',P')$ be one of the chosen blue primary vertices. Then $E'\cong E/E[\fL^\alpha]$ for some nonnegative integer $\alpha$. Proposition~\ref{prop:st} tells us that after applying $\overline{\fL}^t$ to $v$, we obtain a primary blue vertex. Therefore, each cycle $C_i$ contains $h_2/t$ blue primary vertices. In particular, the rest of the vertices on $C_i$ are  green secondary since they lie on a green path linking two blue primary vertices. This results in $h_2-h_2/t$ green secondary vertices on $C_i$. Thus, this gives in total at least $(s-1)(h_2-h_2/t)$ green secondary vertices. But this is exactly the upper bound, hence the equality holds. \end{proof}
 \begin{lemma}
     A blue/green secondary vertex is not central.
 \end{lemma}
 \begin{proof}
     Let $v$ be a blue secondary vertex. Let $v'$ be a green primary vertex connected to $v$ through a blue path of length $1\le a\le s-1$. Let $w$ be a central vertex connected to $v'$ through a green path of length $1\le a'\le t-1$. If $v$ is also central, then $v$ and $w$ are connected through a blue path of length $a''s$. This implies that the isogenies induced by $\mathfrak{L}^a\overline{\mathfrak{L}}^{a'}$ and  $\mathfrak{L}^{a''s}$ coincide on $v$, which is impossible by Lemma~\ref{prop:st}. 
 \end{proof}
By a similar argument, one can show:
 \begin{lemma}
     A blue secondary vertex is not a green primary vertex. And a green secondary one is not blue primary.
 \end{lemma}
 \begin{proof}
     Let $v$ be a blue secondary vertex and let $v'$ be a green primary vertex connected to $v$ through a blue path of length $1\le a\le s-1$. Let $w$ be a central vertex connected to $v'$ via a green path of length $1\le a'\le t-1$. If $v$ is a green primary vertex, then $v$ and $w$ are connected through a green path of length $bt+a''$ with $1\le a''\le t-1$. Thus, the isogeny induced by $\mathfrak{L}^a\overline{\mathfrak{L}}^{a'}$ and the one induced by $\overline{\mathfrak{L}}^{bt+a''}$ coincide, which is impossible by Proposition~\ref{prop:st}. 
 \end{proof}
 \begin{lemma}
     A blue secondary vertex is not blue primary.
 \end{lemma}
 \begin{proof}
     Let $v$ be a blue secondary vertex. Let $w$ and $w'$ be green primary vertices connected through a green path of length $s$ passing through $v$. If $v$ is blue primary, then $w$ and $w'$ are central, which contradicts Lemma~\ref{lem:primary-not-central}. 
 \end{proof}
 \begin{lemma}
     Each blue secondary vertex is a green secondary vertex, and vice versa.
 \end{lemma}
 \begin{proof}
     Let $v$ be  a blue secondary vertex. Let $v'$ be a green primary vertex such that there is a blue path of length $1\le a\le s-1$ from $v$ to $v'$. Let $w$ be a central vertex such that there is a green path of length $1\le a'\le t-1$ form $v'$ to $w$. Let $w'$ be the blue vertex at the end of a blue path of length $s (h_1/s-1)+(s-a)$ starting at $w$. 
     
     There is also a green path of length $t(h_2/t-1)+(t-a')$. The end point $v''$ of this path is a green secondary vertex. Putting these together, we have a path starting at $v$ with $a+s(h_1/s-1)+s-a=h_1$ blue and $h_2$ green edges. Thus, $v''=v'$.
 \end{proof}

\begin{remark}\label{rk:conclude}
To conclude, we may classify the vertices in $\cC_0$ as follows.
\begin{itemize}\setlength\itemsep{0.5em}
    \item We have $h_1/s=h_2/t$ central vertices, which include $v_1$;
    \item The central vertices can be equally distributed  along a blue closed path of length $h_1$ containing $v_1$ (the distance between two consecutive central vertices is $s$). The non-central vertices on this path are  blue primary. There are $\frac{h_1(s-1)}{s}$ such vertices;
    \item The central vertices can also be found along a green closed path  of length $h_1$ containing $v_1$. The non-central vertices on this path are  green primary. There are $\frac{h_2(t-1)}t$ such vertices;
    \item The rest of the vertices are blue secondary, which can be found on a blue path between two primary green vertices. There are  
    \[
    \frac{h_1(s-1)(t-1)}{s}=\frac{h_2(s-1)(t-1)}{t}
    \]
    such vertices;
    \item We may reverse the roles of blue and green in the previous bullet point, resulting in the same vertices.
    \item In total, there are $st\Omega$ vertices, where $\Omega=\frac{h_1}{s}=\frac{h_2}{t}$ is the number of central vertices.\e
\end{itemize}    
\end{remark}

\begin{example}\label{eg12}
Suppose that $h_1=h_2=6$, $s=t=2$ and $c=1$. We have the following graph. The "solid" cycles are obtained  from repeatedly applying $\cL$ and $\overline\cL$  to $v_1$, respectively. This gives us the central vertices  $v_1,v_2$ and $v_3$ (coloured in orange), the blue primary vertices are $v_4,v_5$ and $v_6$, and the green primary vertices are $v_7$, $v_8$ and $v_9$. The "dotted" cycles are the ones obtained from applying $\cL$ and $\overline\cL$ to the primary vertices.  The secondary vertices are $v_{10}$, $v_{11}$ and $v_{12}$ (coloured in black). 

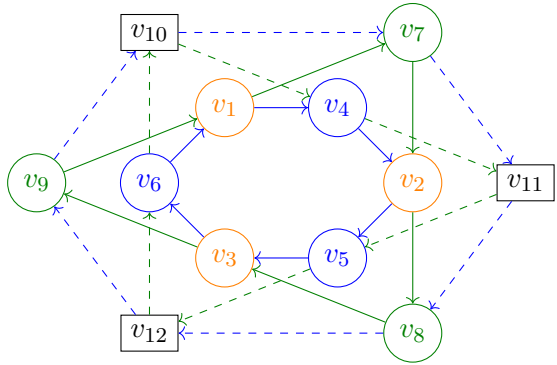
\begin{center}
\begin{tikzpicture}[main/.style = {draw, circle},sec/.style={draw,rectangle}] 
\node[main,orange] at (1,1)(1) {$v_1$}; 
\node[main,orange] (2) at (3.5,0) {$v_2$};
\node[main,orange] (3) at (1,-1) {$v_3$}; 
\node[main,blue] (4) at (2.5,1) {$v_4$};
\node[main,blue] (5) at (2.5,-1) {$v_5$}; 
\node[main,blue]  at (0,0)(6)  {$v_6$};
\node[main,Green] (7) at (3.5,2) {$v_7$};
\node[main,Green] (8) at (3.5,-2) {$v_8$};
\node[main,Green] (9) at (-1.5,0) {$v_9$};
\node[sec] (10)  at (0,2) {$v_{10}$};
\node[sec] (11) at (5,0) {$v_{11}$};
\node[sec] (12) at (0,-2) {$v_{12}$};
\draw[->,blue] (1) -- (4); 
\draw[->,blue] (4) -- (2); 
\draw[->,blue] (2) -- (5); 
\draw[->,blue] (5) -- (3); 
\draw[->,blue] (3) -- (6); 
\draw[->,blue] (6) -- (1); 
\draw[->,dashed,blue] (10) -- (7); 
\draw[->,dashed,blue] (7) -- (11); 
\draw[->,dashed,blue] (11) -- (8);
\draw[->,dashed,blue] (8) -- (12); 
\draw[->,dashed,blue] (12) -- (9); 
\draw[->,dashed,blue] (9) -- (10); 
\draw[->,dashed,Green] (10) -- (4); 
\draw[->,dashed,Green] (4) -- (11); 
\draw[->,dashed,Green] (11) -- (5);
\draw[->,dashed,Green] (5) -- (12);
\draw[->,dashed,Green] (12) -- (6);
\draw[->,dashed,Green] (6) -- (10);
\draw[->,Green] (9) -- (1);
\draw[->,Green] (1) -- (7);
\draw[->,Green] (7) -- (2);
\draw[->,Green] (2) -- (8);
\draw[->,Green] (8) -- (3);
\draw[->,Green] (3) -- (9);
\end{tikzpicture}     
\end{center}

\end{example} 

\begin{example}\label{eg24}
    Suppose $h_1=12$, $h_2=6$, $s=4$, $t=2$ and $c=1$. 
    
    
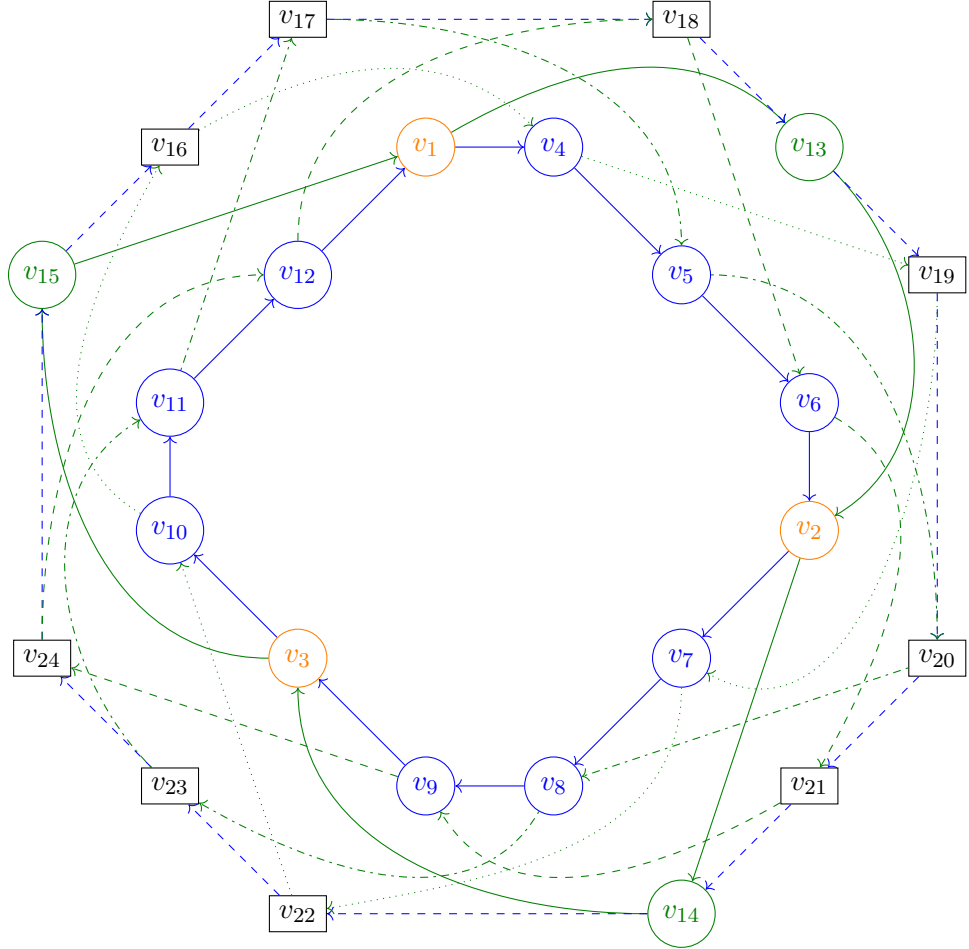
\begin{center}
\begin{tikzpicture}[main/.style = {draw, circle},sec/.style={draw,rectangle},scale=0.85] 
\node[main,orange] at (4,6)(1) {$v_1$}; 
\node[main,orange] (2) at (10,0) {$v_2$};
\node[main,orange] (3) at (2,-2) {$v_3$}; 
\node[main,blue] (4) at (6,6) {$v_4$};
\node[main,blue] (5) at (8,4) {$v_5$}; 
\node[main,blue]  at (10,2)(6)  {$v_6$};
\node[main,blue] (7) at (8,-2) {$v_7$};
\node[main,blue] (8) at (6,-4) {$v_8$};
\node[main,blue] (9) at (4,-4) {$v_9$};
\node[main,blue] (10)  at (0,0) {$v_{10}$};
\node[main,blue] (11) at (0,2) {$v_{11}$};
\node[main,blue] (12) at (2,4) {$v_{12}$};
\node[main,Green] (13) at (10,6) {$v_{13}$};
\node[main,Green] (14) at (8,-6) {$v_{14}$};
\node[main,Green] (15) at (-2,4) {$v_{15}$};
\node[sec] (16) at (0,6) {$v_{16}$};
\node[sec] (17) at (2,8) {$v_{17}$};
\node[sec] (18) at (8,8) {$v_{18}$};
\node[sec] (19) at (12,4) {$v_{19}$};
\node[sec] (20) at (12,-2) {$v_{20}$};
\node[sec] (21) at (10,-4) {$v_{21}$};
\node[sec] (22) at (2,-6) {$v_{22}$};
\node[sec] (23) at (0,-4) {$v_{23}$};
\node[sec] (24) at (-2,-2) {$v_{24}$};

\draw[->,blue] (1) -- (4); 
\draw[->,blue] (4) -- (5);
\draw[->,blue] (5) -- (6);
\draw[->,blue] (6) -- (2); 
\draw[->,blue] (2) -- (7); 
\draw[->,blue] (7) -- (8); 
\draw[->,blue] (8) -- (9); 
\draw[->,blue] (9) -- (3);
\draw[->,blue] (3) -- (10); 
\draw[->,blue] (10) -- (11);
\draw[->,blue] (11) -- (12); 
\draw[->,blue] (12) -- (1); 
\draw[->,Green] (1) to [out=30] (13);
\draw[->,Green] (13) to [in=30,out=-45] (2);
\draw[->,Green] (2) to (14);
\draw[->,Green] (14) to [in=-90,out=-180] (3);
\draw[->,Green] (3) to [in=-90,out=-180] (15);
\draw[->,Green] (15) to  (1);

\newcommand{\bd}{\draw[->,dashed,blue]}
\bd (15) to (16);
\bd (16) to (17);
\bd (17) to (18);
\bd (18) to (13);
\bd (13) to (19);
\bd (19) to (20);
\bd (20) to (21);
\bd (21) to (14);
\bd (14) to (22);
\bd (22) to (23);
\bd (23) to (24);
\bd (24) to (15);
\newcommand{\gd}{\draw[->,dashed,Green]}
\gd (9) to (24);
\gd (24) to [in=-180,out=90] (12);
\gd (12) to [in=-180,out=90](18);
\gd (18) to (6);
\gd (6) to [out=-30,in=60] (21);
\gd (21) to [out=-150,in=-60](9);

\newcommand{\gt}{\draw[->,dotted,Green]}
\gt (16) to [out=30] (4);
\gt (4) to (19);
\gt (19) to [in=-30,out=270] (7);
\gt (7) to [in=10,out=-90] (22);
\gt (22) to (10);
\gt (10) to [out=150,in=-120] (16);

\newcommand{\gD}{\draw[->, dash pattern = on 1 pt off 2 pt on 3 pt off 2 pt ,Green]}
\gD (17) to [out=0,in=90] (5);
\gD (5) to [out=0,in=90] (20);
\gD (20) to (8);
\gD (8) to [out=240,in=330] (23);
\gD (23) to [out=135,in=210] (11);
\gD (11) to (17);
\end{tikzpicture}     
\end{center}

    The central vertices are $v_1$,$v_2$ and $v_3$, which are once again coloured orange. There are $9$ blue primary vertices ($v_4$ to $v_{12}$), $3$ green primary ones ($v_{13}$ to $v_{15}$) and $9$ secondary vertices ($v_{16}$ to $v_{24}$).

\end{example}

\subsection{A special case}
In this section, still assuming $l$ is a split prime,  we specialize to the case where $N=1$  and the ideals of $\cO$ above $l$ are principal. In this case, we can describe the structure of $\cC_0$  more precisely and give a less combinatorial proof.
 \begin{proposition}\label{prop:principal}
          Let $E$ be an elliptic curve of level zero and $P$ a point on $E$ of order $p^m$. Let $\mathcal{O}$ be the endomorphism ring of $E$. Assume that the two ideals above $l$ in $\mathcal{O}$ are principal ideals $\fL=(x)$ and $\overline\fL=(\overline{x})$. Let $\cC_0$ be the connected component of $\cC(G_1^m)$ containing $(E,P)$. Let $V(\cC_0)=\{v_1,\ldots, v_u\}$. Assume that $\cC_0$ does not contain any loops. Then $\cC_0$ satisfies one of the following conditions.
   
         \item[1)] For every $1\le i\le u$, there is an edge from $v_i$ to $v_{i+1}$ (here we consider the indices modulo $u$). Furthermore, there exists  $r\in\{1,2,\ldots, u\}$ such that there is an edge from $v_i$ to $v_{i+r}$ for $1\le i\le u$.
         
         \item [2)] We have $\textup{lcm}(h_1,h_2)=u$, where $h_1$ and $h_2$ are given as in Definition~\ref{def:GNm}. Let $t_i=u/h_i$. Then there is an integer $r$ that is coprime to $u$  such that for $1\le i \le u$ the targets of the edges whose source is $v_i$ are given by  $v_{i+t_1}$ and  $v_{i+rt_2}$, respectively.
    
 \end{proposition}

 \begin{remark}\label{rk:description}
      The order of $x$ (resp. $\overline{x}$) in $(\mathcal{O}/\overline{\mathfrak{p}}^m)^\times/\cO^\times$ is equal to $h_1$ (resp. $h_2$). Here, we have denoted the image of $\cO^\times$ in $(\cO/\overline{\mathfrak p}^m)^\times$ by the same symbol as before.  Suppose that $h_1\ge h_2$. 
      
      We will see in the proof that case 1) occurs when $h_2|h_1$. In this case, $u=h_1$ and the blue edges form the directed cycle $v_1\rightarrow v_2\rightarrow \cdots \rightarrow v_u\rightarrow v_1$. The edges from $v_i$ to $v_{i+r}$ are green. Furthermore, the central vertices are of the form $v_{1+\alpha r}$, $\alpha\in\ZZ$. There are no secondary vertices.
      
      Still assuming $h_1\ge h_2$, if $h_2\nmid h_1$, then case 2) occurs. The edges of the form $v_i\rightarrow v_{i+t_1}$ are blue, whereas those of the form $v_i\rightarrow v_{i+rt_2}$ are green. The central vertices are of the form $v_{1+\alpha rt_1t_2}$, $\alpha\in\ZZ$.\e\end{remark}
 \begin{remark}
     We have seen in the proof of Lemma~\ref{lem:degree-split case} that when $m$ is sufficiently large, the hypothesis that  $\cC_0$ does not admit any loops holds.\e
 \end{remark}
 
\begin{proof} Let $K=\mathcal{O}\otimes \Q$.
    Note that $\textup{Aut}(E[p^m])\cong (\ZZ/p^m\ZZ)^\times$. In particular, the group of automorphisms is cyclic. Let $\mathbf{E}$ be a CM lift of $E$ over $K(j(\mathcal{O}))$ and let $\overline{\mathfrak{p}}$ be a prime ideal above $p$ in $\mathcal{O}$ such that $\mathbf{E}[\overline{\mathfrak{p}}^m]$ reduces to $E[p^m]$ modulo some fixed ideal above $p$ in the ring of integers of $K(j(\mathcal{O}))$. There is a natural ismorphism
    \[\textup{Aut}(E[p^m])\cong (\mathcal{O}/\overline{\mathfrak{p}}^m)^\times.\]

The two horizontal degree $l$ isogenies of $E$ act on the $p^m$-torsion points by $[x]$ and $[\overline{x}]$, respectively. As we have discussed in Remark~\ref{rk:description}, $h_1$ (resp. $h_2$) is the order of $x$ (resp. $\overline{x}$) as an element in $(\mathcal{O}/\overline{\mathfrak{p}}^m)^\times/\mathcal{O}^\times$. Without loss of generality, we assume that $h_1\ge h_2$. 

We first consider the case $h_2\mid h_1$.  As $(\mathcal{O}/\overline{\mathfrak{p}}^m)^\times$ is a cyclic group, there exists an integer $r$ such that the cosets $\mathcal{O}^\times\overline{x}$ and $\mathcal{O}^\times x^r$ in $\cO/ {\overline{\mathfrak{p}}^m}$ coincide. Note that $\textup{gcd}(h_1,r)=h_1/h_2$.

Consider the closed blue path 
\[
C:v_{1}\rightarrow \dots\rightarrow  v_{h_1}\rightarrow v_1
\]
obtained by repeatedly applying $[x]$ to $(E,P)$. If we apply $[\overline x]$ to $v_i$, we obtain a closed green path of the form $v_i\rightarrow v_{i+r}\rightarrow  v_{i+2r}\rightarrow \cdots\rightarrow v_i$. In particular, we see that $C$ contains all vertices of $\cC_0$ and so $u=h_1$ and  $\cC_0$ is described as in 1).

We now consider the case where $h_2$ does not divide $h_1$. Then there exist integers $t_1,t_2>1$ such that $h_1{t_1}=h_2t_2=\textup{lcm}(h_1,h_2)$. Let $$z=\gcd(h_1,h_2)=h_1/t_2=h_2/t_1.$$ 
We see that both $x^{t_2}$ and $\overline x^{t_1}$ have order $z$ as elements in the cyclic group $(\cO/\overline{\mathfrak{p}}^m)^\times/\cO^\times$.  Thus, there exists an integer $r$ coprime to $z$ such that the cosets
$\mathcal{O}^\times \overline{x}^{t_1}$ and $ \mathcal{O}^\times x^{rt_2}$ coincide in $(\cO/\overline{\mathfrak{p}}^m)^\times$.

As $[x]$ and $[\overline{x}]$ commute, each path in $\cC_0$ is given by $[x^a\overline{x}^b]$ for some integers $a$ and $b$. Thus, the number of vertices in $\cC_0$ is given by the cardinality of the subgroup $U$ generated by $x$ and $\overline{x}$ in $(\mathcal{O}/\overline{\mathfrak{p}}^m)^\times/\mathcal{O}^\times$. Let 
\[\Phi\colon \ZZ/h_1\ZZ\times \ZZ/h_2\ZZ\to U,\]
be the surjective group homomorphism given by $(a,b)\mapsto x^a\overline{x}^b$. By the definition of $r$, the kernel of $\Phi$ is generated by $(rt_2,-t_1)$, which generates a cyclic subgroup of $\ZZ/h_1\ZZ\times \ZZ/h_2\ZZ$ of order $z$. It follows that there are in total $u=h_1h_2/z=\textup{lcm}(h_1,h_2)$ vertices in $\cC_0$.

By adding a multiple of $z$ to $r$ if necessary, we may assume that $r$ is coprime to $u$. Consider the group homomorphism
\[\Theta \colon\ZZ/h_1\ZZ\times \ZZ/h_2\ZZ\to \ZZ/u\ZZ, \quad (a,b)\mapsto at_1+brt_2.\]
Since $t_1$ and $rt_2$ are coprime integers,  $\Theta $ is surjective and the kernel is generated by $(rt_2,-t_1)$. Let $(E,P)=v_1$. We enumerate the vertices of $\cC_0$ so that  $v_{\Theta((a,b))}=[x^a\overline{x}^b]v_1$. It then follows that  $[x]$ (resp. $[\overline{x}]$) induces a blue (resp. green) edge from $v_i$ to  $v_{i+t_1}$ (resp.  $v_{i+rt_2}$) as described in 2). 
 \end{proof}

\begin{remark}
    \begin{itemize}
        \item If $x$ and $\overline{x}$ act trivially on $E[p]$, then the order of $x$ and $\overline{x}$ in $\textup{Aut}(E[p^m])$ will always be a $p$-power. In this case $\cC_0$ is  described by 1).
        \item If $m=1$ and $\cC_0$ is described by 2), then $\cC_0$ is described by 2) for all $m\ge1 $.
        \item If $m=2$ and the structure of $\cC_0$ is described by 1), then $\cC_0$ is described by 1) for all $m\ge2$. \e
    \end{itemize}
\end{remark}
\begin{example}
Let $K=\QQ(\sqrt{-5})$ and $p=3$. We consider an elliptic curve $\bE$ with complex multiplication by $\mathcal{O}_K=\ZZ[\sqrt{-5}]$. The two prime ideals above $3$ are $(3, 1+\sqrt{-5})$ and $(3,1-\sqrt{-5})$. We take $\overline{\mathfrak{p}}=(3,1+\sqrt{5})$. Let $l=409$. Then the two ideals above $(409)$ in $\mathcal{O}_K$ are $(2+9\sqrt{-5})$ and $(2-9\sqrt{-5})$. Let $x=2+9\sqrt{-5}$. Then 
\[x\equiv \overline{x}\equiv 2 \pmod {(3,1+\sqrt{-5})^2}.\]
It can be checked that $2$ is indeed a generator of $(\mathcal{O}_K/(3,1+\sqrt{-5})^2)^\times$. It follows that the orders of $x$ and $\overline{x}$ in $(\mathcal{O}_K/(3,1+\sqrt{-5})^m)^\times/\{\pm 1\}$ are $3^{m-1}$. Thus, the graph $\cC_0$ obtained from $\bE\pmod \fp$ is described by case 1) of Proposition~\ref{prop:principal} with $u=3^{m-1}$. \e
\end{example}
\begin{example}
    Let $K=\QQ(\sqrt{-10})$ and $p=13$. We consider again an elliptic curve $\bE$ with complex multiplication by $\mathcal{O}_K=\ZZ+\ZZ\sqrt{-10}$. The two ideals over $13$ are $(13,4+\sqrt{-10})$ and $(13,4-\sqrt{-10})$. Let $\overline{\mathfrak{p}}=(13,4+\sqrt{-10})$. Let $l=11$. The two ideals over $(11)$ are given by $(1+\sqrt{-10})$ and $(1-\sqrt{-10})$. Let $x=1+\sqrt{-10}$. Then \[x\equiv -3\pmod {(13,4+\sqrt{-10})}\] and \[\overline{x}\equiv 5\pmod {(13,4+\sqrt{-10})}.\] The order of $-3$ in $(\mathcal{O}_K/(13,4+\sqrt{-10}))^\times /\{\pm 1\}$ is $3$, while the order of $5$ is $2$. It follows that for $m=1$, the graph $\cC_0$ obtained from $\bE\pmod {\overline{\fp}}$ is described by case 2) of Proposition~\ref{prop:principal} with $u=6$, $h_1=3$ and $h_2=2$.  \e
\end{example}

We conclude this section with a number of illustrations of several graphs given by Proposition~\ref{prop:principal}.

\subsection*{Case 1) with $u=5$} Since we have assumed that there is no loop, we have $h_1,h_2>1$. Thus, in order for case 1) to occur, we must have $h_1=h_2=5$. Every vertex is central and we have the complete graph $K_5$ (after ignoring the directions). Depending on the value of $r$, we have one of the following two graphs.
 
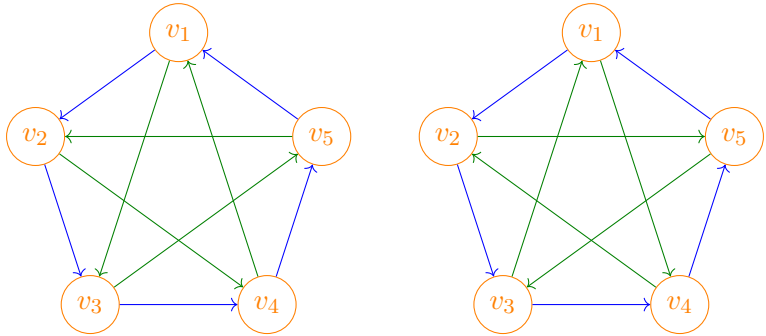
\begin{center}\begin{tikzpicture}[vertex/.style={circle, draw}] 
\foreach \X[count=\Y] in {blue,red,green,yellow,green}
{\node[draw,circle,orange] (x-\Y) at ({72*\Y+18}:2){$v_\Y$}; }
\foreach \X[count=\Y] in {0,...,4}
{\ifnum\X=0
\draw[->,blue] (x-5) --(x-\Y) ;
\else
\draw[->,blue] (x-\X) --(x-\Y) ;
\fi}
\draw[->,Green] (x-1)--(x-3);
\draw[->,Green] (x-3)--(x-5);
\draw[->,Green] (x-5)--(x-2);
\draw[->,Green] (x-2)--(x-4);
\draw[->,Green] (x-4)--(x-1);
\end{tikzpicture}
\qquad
\begin{tikzpicture}[vertex/.style={circle, draw}]
\foreach \X[count=\Y] in {blue,red,green,yellow,green}
{\node[draw,circle,orange] (x-\Y) at ({72*\Y+18}:2){$v_\Y$}; }
\foreach \X[count=\Y] in {0,...,4}
{\ifnum\X=0
\draw[->,blue] (x-5) --(x-\Y) ;
\else
\draw[->,blue] (x-\X) --(x-\Y) ;
\fi}
\draw[->,Green] (x-1)--(x-4);
\draw[->,Green] (x-4)--(x-2);
\draw[->,Green] (x-2)--(x-5);
\draw[->,Green] (x-5)--(x-3);
\draw[->,Green] (x-3)--(x-1);
\end{tikzpicture}
 \end{center}


\subsection*{Case 1) with $u=6$} We have $h_1=6$. We illustrate the cases where $h_2=3$  and $h_2=2$ below. When $h_2=3$, $r$ can be either $2$ or $4$.

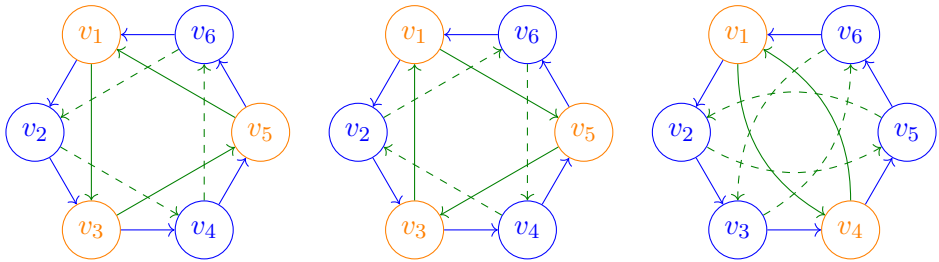
\begin{center}
\begin{tikzpicture}[vertex/.style={circle, draw}] 
\foreach \X[count=\Y] in {orange,blue,orange,blue,orange,blue}
{\node[draw,circle,\X] (x-\Y) at ({60*\Y+60}:1.5){$v_\Y$}; }
\foreach \X[count=\Y] in {0,...,5}
{\ifnum\X=0
\draw[->,blue] (x-6) --(x-\Y) ;
\else
\draw[->,blue] (x-\X) --(x-\Y) ;
\fi}
\draw[->,Green] (x-1)--(x-3);
\draw[->,Green] (x-3)--(x-5);
\draw[->,Green] (x-5)--(x-1);
\draw[->,dashed,Green] (x-2)--(x-4);
\draw[->,dashed,Green] (x-4)--(x-6);
\draw[->,dashed,Green] (x-6)--(x-2);
\end{tikzpicture}
\quad
\begin{tikzpicture}[vertex/.style={circle, draw}] 
\foreach \X[count=\Y] in {orange,blue,orange,blue,orange,blue}
{\node[draw,circle,\X] (x-\Y) at ({60*\Y+60}:1.5){$v_\Y$}; }
\foreach \X[count=\Y] in {0,...,5}
{\ifnum\X=0
\draw[->,blue] (x-6) --(x-\Y) ;
\else
\draw[->,blue] (x-\X) --(x-\Y) ;
\fi}
\draw[->,Green] (x-1)--(x-5);
\draw[->,Green] (x-5)--(x-3);
\draw[->,Green] (x-3)--(x-1);
\draw[->,dashed,Green] (x-2)--(x-6);
\draw[->,dashed,Green] (x-6)--(x-4);
\draw[->,dashed,Green] (x-4)--(x-2);
\end{tikzpicture}
\quad
\begin{tikzpicture}[vertex/.style={circle, draw}] 
\foreach \X[count=\Y] in {orange,blue,blue,orange,blue,blue}
{\node[draw,circle,\X] (x-\Y) at ({60*\Y+60}:1.5){$v_\Y$}; }
\foreach \X[count=\Y] in {0,...,5}
{\ifnum\X=0
\draw[->,blue] (x-6) --(x-\Y) ;
\else
\draw[->,blue] (x-\X) --(x-\Y) ;
\fi}
\draw[->,Green] (x-1) to [bend right] (x-4);
\draw[<-,Green] (x-1) to [bend left] (x-4);
\draw[->,dashed,Green] (x-5) to [bend right] (x-2);
\draw[<-,dashed,Green] (x-5) to [bend left] (x-2);
\draw[->,dashed,Green] (x-3) to [bend right] (x-6);
\draw[<-,dashed,Green] (x-3) to [bend left] (x-6);
\end{tikzpicture} \end{center}

\subsection*{Case 1) with $u=12$}  Suppose that $h_1=12$, $h_2=4$ with $r=3$. We have the following graph.
 
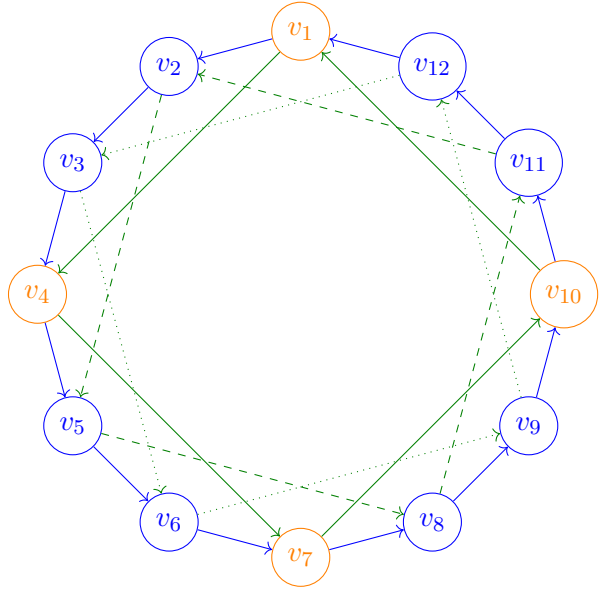
\begin{center}
     \begin{tikzpicture}[vertex/.style={circle, draw}] 
\foreach \X[count=\Y] in {orange,blue,blue,orange,blue,blue,orange,blue,blue,orange,blue,blue}
{\node[draw,circle,\X] (x-\Y) at ({30*\Y+60}:3.5){$v_{\Y}$}; }
\foreach \X[count=\Y] in {0,...,11}
{\ifnum\X=0
\draw[->,blue] (x-12) --(x-\Y) ;
\else
\draw[->,blue] (x-\X) --(x-\Y) ;
\fi}
\draw[->,Green] (x-1)--(x-4);
\draw[->,Green] (x-4)--(x-7);
\draw[->,Green] (x-7)--(x-10);
\draw[->,Green] (x-10)--(x-1);
\draw[->,dashed,Green] (x-2)--(x-5);
\draw[->,dashed,Green] (x-5)--(x-8);
\draw[->,dashed,Green] (x-8)--(x-11);
\draw[->,dashed,Green] (x-11)--(x-2);
\draw[->,dotted,Green] (x-3)--(x-6);
\draw[->,dotted,Green] (x-6)--(x-9);
\draw[->,dotted,Green] (x-9)--(x-12);
\draw[->,dotted,Green] (x-12)--(x-3);
\end{tikzpicture}
 \end{center}
If $r=9$, then the directions of the green edges are reversed. 

\subsection*{Case 2) with $u=6$}  Suppose that $h_1=3$, $h_2=2$ and $r=1$. Then $t_1=2$ and $t_2=3$. We have the following graph:

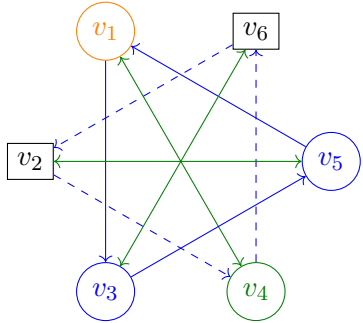
\begin{center}
 \begin{tikzpicture}[vertex/.style={circle, draw}] 

\node[draw,circle,orange] (x-1) at  (120:2){$v_1$};
\node[draw,rectangle,black] (x-2) at  (180:2){$v_2$};
\node[draw,circle,blue] (x-3) at  (240:2){$v_3$};
\node[draw,circle,Green] (x-4) at  (300:2){$v_4$};
\node[draw,circle,blue] (x-5) at  (360:2){$v_5$};
\node[draw,rectangle,black] (x-6) at  (60:2){$v_6$};

\draw[<->,Green] (x-1)--(x-4);
\draw[->,blue] (x-1)--(x-3);
\draw[->,blue] (x-3)--(x-5);
\draw[->,blue] (x-5)--(x-1);
\draw[->,dashed,blue] (x-2)--(x-4);
\draw[->,dashed,blue] (x-4)--(x-6);
\draw[->,dashed,blue] (x-6)--(x-2);
\draw[<->,Green] (x-2)--(x-5);
\draw[<->,Green] (x-3)--(x-6);
\end{tikzpicture}
\end{center}

\subsection*{Case 2) with $u=12$} Suppose that $h_1=4$, $h_2=3$ and $r=1$. Then $t_1=3$, $t_2=4$, resulting in the following graph:

\begin{center}
     \begin{tikzpicture}[vertex/.style={circle, draw}]

\node[draw,circle,orange] (x-1) at  (90:4){$v_1$};
\node[draw,rectangle,black] (x-2) at  (120:4){$v_2$};
\node[draw,rectangle,black] (x-3) at  (150:4){$v_3$};
\node[draw,circle,blue] (x-4) at  (180:4){$v_4$};
\node[draw,circle,Green] (x-5) at  (210:4){$v_5$};
\node[draw,rectangle,black] (x-6) at  (240:4){$v_6$};
\node[draw,circle,blue] (x-7) at  (270:4){$v_7$};
\node[draw,rectangle,black] (x-8) at  (300:4){$v_8$};
\node[draw,circle,Green] (x-9) at  (330:4){$v_9$};
\node[draw,circle,blue] (x-10) at  (360:4){$v_{10}$};
\node[draw,rectangle,black] (x-11) at  (30:4){$v_{11}$};
\node[draw,rectangle,black] (x-12) at  (60:4){$v_{12}$};

\draw[->,blue] (x-1)--(x-4);
\draw[->,blue] (x-4)--(x-7);
\draw[->,blue] (x-7)--(x-10);
\draw[->,blue] (x-10)--(x-1);
\draw[->,Green] (x-1)--(x-5);
\draw[->,Green] (x-5)--(x-9);
\draw[->,Green] (x-9)--(x-1);

\draw[->,dotted,blue] (x-2)--(x-5);
\draw[->,dotted,blue] (x-5)--(x-8);
\draw[->,dotted,blue] (x-8)--(x-11);
\draw[->,dotted,blue] (x-11)--(x-2);
\draw[->,dotted,Green] (x-2)--(x-6);
\draw[->,dotted,Green] (x-6)--(x-10);
\draw[->,dotted,Green] (x-10)--(x-2);

\draw[->,dashed,blue] (x-3)--(x-6);
\draw[->,dashed,blue] (x-6)--(x-9);
\draw[->,dashed,blue] (x-9)--(x-12);
\draw[->,dashed,blue] (x-12)--(x-3);
\draw[->,dashed,Green] (x-3)--(x-7);
\draw[->,dashed,Green] (x-7)--(x-11);
\draw[->,dashed,Green] (x-11)--(x-3);

\draw[->, dash pattern = on 1 pt off 2 pt on 3 pt off 2 pt ,Green] (x-4)--(x-8);
\draw[->, dash pattern = on 1 pt off 2 pt on 3 pt off 2 pt ,Green] (x-8)--(x-12);
\draw[->, dash pattern = on 1 pt off 2 pt on 3 pt off 2 pt ,Green] (x-12)--(x-4);

\end{tikzpicture}
 \end{center}

\section{An inverse problem for craters}\label{S:inverse}
The goal of this section is to prove Theorem~\ref{thmB} stated in the introduction. In particular, we study an inverse problem for the graphs arising in \S\ref{S:split-crater}. We first introduce the following  definition of graphs, which can be regarded as a partial generalization of abstract volcano graphs introduced in \cite[Definition~4.1]{pazuki}.
\begin{defn}\label{def:crater}
  Let $\r,\s,\t,\cc$ be nonnegative integers. We say that a directed graph is an \textbf{abstract tectonic crater} of parameters $(\r,\s,\t,\cc)$ if it satisfies
    \begin{itemize}
  \item[a)] There are $\r\s\t$ vertices;
  \item[b)] Each edge is assigned a color -- blue or green;
  \item[c)] At each vertex $v$, there is exactly one blue edge with $v$ as the source, and exactly one blue edge with $v$ as the target, and similarly for green edges;
  \item[d)] Starting at each vertex, there is exactly one closed blue (resp. green) path without backtracks of length $\r\s$ (resp. $\r\t$);
  \item[e)] After every $\s$ (resp. $\cc\t$) steps in the the closed blue (resp. green) paths given in d), the two paths meet at a common vertex.  
  \end{itemize}
\end{defn}
We now prove Theorem~\ref{thmB}.
\begin{theorem}\label{thm:inverse}
Let  $G$ be an abstract tectonic crater. There exist infinitely many pairs of distinct primes $p $ and $l$, and nonnegative integers $N$ such that one of the connected components of the crater of the $l$-isogeny graph $G_N^1$ (over $\mathbb{F}_p$) is isomorphic to $G$.
\end{theorem}
\begin{remark}
    We emphasize that an abstract tectonic crater can never describe a connected component of $G_1^0$. Indeed, each vertex $v$ in an abstract tectonic admits $4$ edges, two with source $v$ and to with target $v$. In $G_1^0$ each vertex admits at most two vertices. Thus, the level structure is crucial for the above result.
\end{remark}
\begin{proof}
Let $(\Omega,s,t,c)$ be the parameters of $G$. We shall construct a $\cC_0$ that is isomorphic to $G$ where the symbols $\Omega,s,t,c$ have the same significations as those assigned in Remark~\ref{rk:conclude}.

Let $K$ be an imaginary quadratic field different from $\Q(\sqrt{-1})$ and $\Q(\sqrt{-3})$. Let $\mathfrak{f}$ be an ideal coprime to $10$ and let $F=K(\mathfrak{f})$ be the ray class field of conductor $\mathfrak f$. Then there exists an elliptic curve $\mathbf{E}/F$ that has complex multiplication by $\mathcal{O}_K$ and such that $K(\mathbf{E}_\textup{tors})/K$ is abelian (see \cite[Lemma 2]{crisan-mueller} or \cite[Chapter II, 1.4]{shalit}). Let $L'=F(\mathbf{E}[5])$. Then, the theory of complex multiplication tells us that $\Gal(L'/K)\cong (\mathcal{O}_K/(5))^\times /\{\pm 1\}$. 

Let $p$ be a prime number such that
\begin{itemize}
    \item $p\ne 5$;
    \item $p\equiv 1\pmod \Omega$;
    \item $p$ is totally split in $F$.
\end{itemize}
 Then $\mathbf{E}$ has good ordinary reduction at all the primes of $F$ lying above $p$. We fix once and for all a prime $\mathfrak{p}'|p$ of $F$. Let $\mathfrak{p}$ be the unique prime of $K$ lying below $\fp'$. 

Choose  two different prime numbers $N'$ and $M'$ that are coprime to $5p\mathfrak{f}$ and are split in $K$, with $N'\equiv 1\pmod s$ and $M'\equiv 1\pmod t$ . Let $\mathfrak{N}$ (resp.  $\mathfrak{M}$) be a prime ideal of $\mathcal{O}_K$ lying above $N'$ (resp. $M'$). Let $L=F(\bE[5N'M'])$. Define furthermore $L_1=F(\bE[5N'M'\overline{\mathfrak{p}}])$ and $L_2=F(\bE[5N'M'\mathfrak{p}])$. We have the following group isomorphisms
\begin{equation}\label{eq:Gal1}
    \Gal(L_1/L')\cong(\mathcal{O}_K/\mathfrak{N})^\times \times (\mathcal{O}_K/\overline{\mathfrak{N}})^\times \times (\mathcal{O}_K/\mathfrak{M})^\times \times (\mathcal{O}_K/\overline{\mathfrak{M}})^\times\times (\mathcal{O}_K/\overline{\mathfrak{p}})^\times,
\end{equation}
and 
\begin{equation}\label{eq:Gal2}
    \Gal(L_2/L')\cong(\mathcal{O}_K/\mathfrak{N})^\times \times (\mathcal{O}_K/\overline{\mathfrak{N}})^\times \times (\mathcal{O}_K/\mathfrak{M})^\times \times (\mathcal{O}_K/\overline{\mathfrak{M}})^\times\times (\mathcal{O}_K/\mathfrak{p})^\times.
\end{equation}
Furthermore $L_1\bigcap L_2=L$.

By Tchebotarev's theorem, there exists a prime ideal $\mathfrak{L}$ in $\mathcal{O}_K$ such that 
\begin{itemize}
    \item[i)] $\mathfrak{L}$ splits in $L'/K$;
    \item[ii)] $\mathfrak{L}\neq \overline{\mathfrak{L}}$;
    \item[iii)] The Frobenius of $\mathfrak{L}$ in $\Gal(L_1/L')$ gives rise to an element of the form $(a,1,1,b,d)$ on the right-hand side of \eqref{eq:Gal1}, where $\ord(a)=s$, $\ord(b)=t$ and $\ord(d^s)=\Omega$;
    \item[iv)] The Frobenius of $\mathfrak{L}$ in $\Gal(L_2/L')$ gives rise to an element of the form $(a,1,1,b,d')$ on the right-hand side of \eqref{eq:Gal2}, with $\ord(d')^t=\Omega$ and ${d'}^{ct}=d^s$ (after identifying $\mathcal{O}/\overline{\mathfrak{p}}$ with $\mathcal{O}/\mathfrak{p}$ via complex conjugation).
\end{itemize}
Let $\sigma$ (resp. $\tau$) the Frobenius of $\fL$ (resp. $\overline{\fL}$) in $\Gal(L_1/L)$. Note that $\tau$ gives an element of  the form $(1,a,b,1,d')$ on the right-hand side of \eqref{eq:Gal1}.

Let $H_1$ (resp. $H_2$) be the cyclic subgroup of $\Gal(L_1/L')$ generated by $\sigma$ (resp. $\tau$). Let $Q$ be a primitive $5\mathfrak{N}\mathfrak{M}$-torsion point on $\bE$. Then $\sigma$ (resp. $\tau$) acts on $Q$ via $(a,1,1,1,1)$ (resp. $(1,1,b,1,1)$). The orbit of $Q$ under the action of $H_1$ (resp. $H_2$) contains $s$ (resp. $t$) elements. As $\sigma$ and $\tau$ fix $\bE[5]$ by construction, $-Q$ is not contained in either of these orbits and the only point contained in both orbits is $Q$.

Now, let $P$ be a primitive $5\mathfrak{N}\mathfrak{M}\overline{\mathfrak{p}}$-torsion point in $\bE$ such that $pP=Q$. By the conditions iii) and iv),  the orbit of $P$ under the action of $H_1$ contains $s\Omega$ elements, whereas that  under $H_2$ contains $t\Omega$ elements. Again neither of these orbits contains $-P$. Note that $\sigma^s$ is of the form $(a^s,1,1,b^s,d^s)$, while $\tau^{ct}$ is of the form $(1,a^{ct},b^{ct},1,{d'}^{ct})$ by construction, and both elements act on $P$ via $(1,1,1,d^s)$. Therefore, $\sigma^s(P)=\tau^{ct}(P)$ and the orbits of $P$ under $H_1$ and $H_2$ intersect precisely in the set $\{P,\sigma^s P,\sigma^{2s}P,\dots, \sigma^{(\Omega-1)s}P\}$.

We have $\sigma(P)=\phi(P)$, where $\phi$ is the isogeny $\bE\to \bE/\bE[\mathfrak{L}]\cong \bE$ and likewise for $\tau$ (see \cite[Chapter II, 1.3 and 1.4]{shalit}). Let $E=\mathbf
{E}\pmod {\mathfrak{p}'}$, $l$ be the rational prime below $\fL$ and $N=5N'M'$. Then the connected component $\cC_0$ containing $(E,P\pmod {\mathfrak{p}'})$ is precisely the tectonic crater $G$.
\end{proof}

 \begin{remark}
     Let $u=\mathrm{lcm}(h_1,h_2)$ and suppose that there are  $u$ vertices in $\cC_0$. Let $V=\{v_1,\dots,v_u\}$ be the set of vertices. After relabelling if necessary, we have blue edges going from $v_i$ to $v_{i+u/h_1}$ and green edges from $v_i$ to $v_{i+c\cdot u/h_2}$ for some $c$ that is coprime to $u$. 
     
     We have seen in Remark~\ref{rk:conclude} that the total number of vertices is given by $h_2s=h_1t$. Thus, we have $$u=\textup{lcm}(h_1,h_2)=h_2s=h_1t.$$ This happens if the smallest subgroup of $\textup{Aut}(V)$ containing $H_1$ and $H_2$ is cyclic (where $H_1$ and $H_2$ are the subgroups defined in the proof of Theorem~\ref{thm:inverse}). If $H$ has rank two,  the graph contains two blue directed cycles, as in Examples \ref{eg12} and \ref{eg24}.\e
 \end{remark}

\appendix
\section{Voltage assignment}\label{app}
The goal of this appendix is to describe the voltage assignment associated with the tower $(G_1^m)_{m\ge0}$.  Let us recall the definition of a voltage assignment on a directed graph.

\begin{defn}
    Let $X$ be a directed graph, $(G,\cdot)$ an abelian group and $n\ge1$ an integer. A $G$-valued \textbf{voltage assignment} on $X$ is a function $\alpha:\EE(X)\rightarrow G$.

    To each voltage assignment, we define the derived graph $X(G,\alpha)$ whose vertices and edges are given by $V(X)\times G$ and $\EE(X)\times G$ respectively. If  $(e,\sigma)\in \EE(X)\times G$, it links $(s,\sigma)$ to $(t,\sigma\cdot\alpha(e))$,  where $e$ is an edge in $X$ from $s$ to $t$.
\end{defn}
Note that $X(G,\alpha)\rightarrow X$ given by $(x,\sigma)\mapsto x$ is a graph covering.

Let $X=G_1^0$. For each $E\in \cE$, we fix a group isomorphism
\[
\Phi_{E}:E[p^\infty]\rightarrow  \Qp/\Zp.
\]
This is equivalent to fixing a $\Zp$-basis of the $p$-adic Tate module $T_p(E)$. More explicitly,  let $t_E$ be such a basis. Given $P\in E[p^m]$, we have $\Phi_E(P)=a/p^m+\Zp$ for a unique integer $a\in\{0,1,\ldots, p^m-1\}$.
\begin{equation}\label{eq:identify-torsion}
    \Phi_E(P)=a \overline{t_E},
\end{equation}
where $\overline{t_E}$ denotes the image of $t_E$ in $E[p^m]$.

Let us write $Z_m=\frac{1}{p^m}\Zp/\Zp$ and $Z_m^*=Z_m\setminus Z_{m-1}$, which we may identify with $(\ZZ/p^m\ZZ)^\times$.
Then, we may identify the vertices of $G_1^m$ with $G_1^0\times Z_m^*$, given by $(E,P)\mapsto (E,\Phi_E(P))$. Under \eqref{eq:identify-torsion}, $\Phi_E(P)$ is given by the image of $a$ in $(\ZZ/p^m\ZZ)^\times$.

\begin{defn}
    Let $E_1,E_2\in\cE$. Suppose that there exists an $l$-isogeny $\phi:E_1\rightarrow E_2$. For each equivalence class of degree $l$-isogenies we fix one representative $\phi$.
    We write $t_\phi\in\Zp^\times$ to be the unique element such that
    \begin{equation}\label{eq:phi-basis}
        \phi^*(t_{E_1})=t_\phi\cdot t_{E_2}, 
    \end{equation}where $\phi^*:T_p(E_1)\rightarrow T_p(E_2)$ is the $\Zp$-isomorphism induced by $\phi$.
\end{defn}

Let $\phi:E\rightarrow E'$ be an $\ell$-isogeny. This induces an edge $e$ in $G_1^0$. Consider the corresponding edge from $(E,P)$ to $(E',P')$ in $G_1^m$. Then, one can check that
\[
\Phi_E(P)t_\phi=\Phi_{E'}(P').
\]
Therefore, we may identify $G_1^m$ with the voltage graph $X((\ZZ/p^m\ZZ)^\times,\alpha)$, where $\alpha$ is the voltage assignment $\alpha$ on $G_1^m$ sending $\phi$ to $t_\phi\pmod{p^m}$.

\begin{remark}
    While our voltage assignment depends on a choice of basis for each $T_p(E)$ as $E$ runs through $\cE$. This can be regarded as the analogue of picking a spanning tree of $G_1^0$ (when it is connected) as in \cite[Theorem~2.11]{DLRV}. Indeed, suppose that $G_1^0$ is connected and $\cT$ is a spanning tree. Since $l\ne p$, an $l$-isogeny $\phi:E_1\rightarrow E_2$ induces an isomorphism $\phi^*:T_p(E_1)\rightarrow T_p(E_2)$. Thus, once a basis is picked for one $E\in\cE$, this basis can be propagated to all other curves in $\cE$ along $\cT$.\e
\end{remark}

\begin{remark}
    Since we have realized $G_1^m$ as a voltage graph arising from a voltage assignment on $G_1^0$, this gives an alternative proof of Lemma~\ref{lem:cover} in the special case where $N=1$.\e
\end{remark}

\bibliographystyle{amsalpha}
\bibliography{references}

\end{document}